\documentclass[11pt]{amsart}

\usepackage{amsthm}
\usepackage{amsmath}
\usepackage{amssymb}
\usepackage{color}

\newcommand{\tr}{\mathop\mathrm{tr}}

\newcommand{\spn}{\mathop\mathrm{span}}

\newtheorem{thm}{Theorem}[section]
\newtheorem{theorem}[thm]{Theorem}

\newtheorem{corollary}[thm]{Corollary}

\newtheorem{example}[thm]{Example}

\newtheorem{lemma}[thm]{Lemma}
\newtheorem{proposition}[thm]{Proposition}

\newtheorem{definition}[thm]{Definition}
\theoremstyle{remark}
\newtheorem{remark}[thm]{Remark}

\newcommand{\RR}{\mathbb R}

\numberwithin{equation}{section}
\begin{document}

\title{Regular two-distance sets}
\author[Casazza, Tran, Tremain
 ]{Peter G. Casazza, Tin T. Tran, and Janet C. Tremain}
\address{Department of Mathematics, University
of Missouri, Columbia, MO 65211-4100}

\thanks{The authors were supported by
 NSF DMS 1609760 and 1906725.}

\email{Casazzap@missouri.edu}
\email{tinmizzou@gmail.com}
\email{Tremainjc@missouri.edu}

\begin{abstract}
This paper makes a deep study of regular two-distance sets. A set of unit vectors $X$ in Euclidean space $\RR^n$ is said to be regular two-distance set if the inner product of any pair of its vectors is either $\alpha$ or $\beta$, and the number of $\alpha$ (and hence $\beta$) on each row of the Gram matrix of $X$ is the same. We present various properties of these sets as well as focus on the case where they form tight frames for the underling space. We then give some constructions of regular two-distance sets, in particular, two-distance frames, both tight and non-tight cases. It has been seen that every known example of maximal two-distance sets are tight frames. However, we supply for the first time an example of a non-tight maximal two-distance frame. Connections among two-distance sets, equiangular lines and quasi-symmetric designs are also discussed. For instance, we give a sufficient condition for constructing sets of equiangular lines from regular two-distance sets, especially from quasi-symmetric designs satisfying certain conditions.

\end{abstract}

\maketitle
\pagebreak

\pagebreak

\section{Introduction and Preliminaries}
A set $X$ in Euclidean space $\RR^n$ is called a {\it two-distance set} if there are two numbers $a$ and $b$ such that the distances between any pairs of points of $X$ are either $a$ or $b$. If a two-distance set $X$ lies in the unit sphere of $\RR^n$, then $X$ is called a spherical two-distance set. In other words, a set of unit vectors in $n$-dimensional Euclidean space is a spherical two-distance set if there are two real numbers $\alpha$ and $\beta$, $-1\leq \alpha, \beta\leq 1$ such that the inner product of any two vectors of $X$ are either $\alpha$ or $\beta$. We will say that $\alpha$ and $\beta$ are the angles of $X$.

Studying the maximum size $g(n)$ of a spherical two-distance set of distinct vectors in $\RR^n$ is a classical problem in distance geometry. The first major result was obtained in \cite{DJS}, where the authors showed the ``harmonic'' bound:
\begin{equation}\label{eneq1}
g(n)\leq \frac{n(n+3)}{2}.
\end{equation}
Moreover, they showed that this bound is achieved when $n=2, 6, 22$, in which cases it related to the maximal set of equiangular lines in dimension $n+1$. The result of \cite{DJS} also showed that this bound can be attained only if $n=(2k+1)^2-3$ for $k\in \mathbb{N}$, $n>2$.

For $n<7$, it is known that $g(2)=5, g(3)=6, g(4)=10, g(5)=16$ and $g(6)=27$, see \cite{L, M}.  In \cite{M}, Musin showed that the size of a spherical two-distance set of distinct vectors in $\RR^n$ with angles $\alpha, \beta$ satisfying $\alpha+\beta\geq0$ is not greater than $\frac{n(n+1)}{2}$. The author also extended the maximum bound $g(n)$ for $n<40, n\not=22, 23.$

Recently, using the results of Musin and combining the known bounds for $n<359$, Glazyrin and Yu made a serious advance when they showed that
\begin{equation}\label{eneq2}
g(n)=\frac{n(n+1)}{2},
\end{equation}
 for all $n\geq 7$ with possible exceptions for $n=(2k+1)^2-3, k\in \mathbb{N}$, see \cite{GY}.

 In this paper, we study regular two-distance sets, in particular two-distance tight frames. These sets are special cases of spherical two-distance sets. Before giving the definition, let us fix some notations used in the paper. 
 
 For any natural number $m$, we denote by $[m]$ the set $[m]=\{1, 2, \ldots, m\}$. We use $\bf{1}$ to denote the vector of all 1, $I$ the identity matrix, and $J$ the matrix whose all its entries are 1. The order of these matrices are always clear from context. For a set of vectors $\{x_i\}_{i=1}^m$ in $\RR^n$, its {\it Gram matrix} is the $m$ by $m$ matrix with entries $G_{ij}=\langle x_i, x_j\rangle$ for $i, j\in [m]$.

Let $\{x_i\}_{i=1}^m$ be a spherical two-distance set in $\RR^n$ at angles $\alpha$ and $\beta$. For each $i\in [m]$, we define the sets
\[\mathcal{I}^\alpha_i=\{j\in [m]: \langle x_i, x_j\rangle=\alpha\}, \quad \mathcal{I}^\beta_i=\{j\in [m]: \langle x_i, x_j\rangle=\beta\}.\]
It is clear that $|\mathcal{I}^\alpha_i|+|\mathcal{I}^\beta_i|=m-1$ for all $i\in [m]$. 

In general, the cardinalities of these sets, $\mathcal{I}^\alpha_i$ and $\mathcal{I}^\beta_i$, depend on $i$. When they are independent with $i$, we say that the set is {\it regular}.
\begin{definition}
	A spherical two-distance set $X=\{x_i\}_{i=1}^m$ in $\RR^n$ at angles $\alpha$ and $\beta$ is said to be regular if the cardinality of the set $\mathcal{I}^\alpha_i$ (and hence the set $\mathcal{I}^\beta_i)$ does not depend on $i$. We will call the numbers $k_\alpha:=|\mathcal{I}^\alpha_i|$, and $k_\beta:=|\mathcal{I}^\beta_i|$, the multiplicities of $\alpha$ and $\beta$, respectively.
\end{definition}
The following theorem will give a simple characterization of regular two-distance sets.
\begin{theorem}\label{thm1}
	A spherical two-distance set is regular if and only if its Gram matrix has constant row sum.
\end{theorem}
\begin{proof}
		Assume $X=\{x_i\}_{i=1}^m$ is a spherical two-distance set with angles $\alpha$ and $\beta$. Let $G$ be its Gram matrix. If $X$ is regular, then obviously $G$ has constant row sum. 
		
		Conversely, assume that the Gram matrix $G$ has constant row sum $c$. Then 
		\begin{align*}
		\sum_{j=1}^{m}\langle x_i, x_j\rangle&=1+|\mathcal{I}^\alpha_i|\alpha+|\mathcal{I}^\beta_i|\beta\\
		&=1+|\mathcal{I}^\alpha_i|\alpha+(m-1-|\mathcal{I}^\alpha_i|)\beta=c, \mbox{ for all } i\in [m].
		\end{align*}
	
		Therefore, for any $i\not=\ell$, we have
		\[(|\mathcal{I}^\alpha_i|-|\mathcal{I}^\alpha_\ell|)\alpha-(|\mathcal{I}^\alpha_i|-|\mathcal{I}^\alpha_\ell|)\beta=0,\]
		or equivalently,
		\[(|\mathcal{I}^\alpha_i|-|\mathcal{I}^\alpha_\ell|)(\alpha-\beta)=0.\]
		Since $\alpha\not=\beta$, we get $|\mathcal{I}^\alpha_i|=|\mathcal{I}^\alpha_\ell|$, which is the desired claim.
\end{proof}

Thus, for a regular two-distance set $X$, the sum of the entries in every row of its Gram matrix are the same. We will call this common number the {\bf Grammian constant} of $X$.
It is clear that this constant is always greater than or equal to zero and less than the cardinality of $X$.

Frames have been shown very useful in variety of applications, see the books \cite{CG, W} and references therein. Therefore, in this paper, we are also interested in the case where spherical two-distance sets form frames for the underlining spaces. The following are some basic facts of frame theory. For further background on finite frame theory, we recommend the books \cite{CG, W}.

\begin{definition}
	A family of vectors $X=\{x_i\}_{i=1}^m$ in $\RR^n$ is said to be a frame for $\RR^n$ if there are constants $0 < A \leq B < \infty$ so that for all $x\in \RR^n$ we have
	\[A\|x\|^2\leq \sum_{i=1}^{m}|\langle x, x_i\rangle|^2\leq B\|x\|^2.\]
	$A$ and $B$ are called the lower and upper frame bounds, respectively.
	The frame is called an $A$-tight frame if A = B and a Parseval frame if
	$A = B = 1$.
\end{definition}
It is well-known that $X$ is a frame for $\RR^n$ if and only if it spans the space. Given a frame $X=\{x_i\}_{i=1}^m$ for $\RR^n$, the corresponding synthesis operator, also denoted by $X$, is the $n\times m$ matrix whose $j$th column is $x_j$. The adjoint matrix $X^*$ is called the analysis operator, and the frame operator of $X$ is then $S:=XX^*$. Thus, we have 
\[Sx=\sum_{i=1}^{m}\langle x, x_i\rangle x_i, \mbox{ for all } x\in \RR^n.\]  $X$ is an $A$-tight frame if and only if its frame operator $S$ is a multiple of identity, namely $S=A.I$. In other words, $X$ is an $A$-tight frame for $\RR^n$ if it satisfies the reconstruction formula:
\[Ax=\sum_{i=1}^{m}\langle x, x_i\rangle x_i,  \mbox{ for all } x\in \RR^n.\]

When all vectors of an $A$-tight frame of $m$ vectors for $\RR^n$ are unit norm, it is known that $A=m/n$. Note also that $X^*X$ is the Gram matrix of $X$.

Another important characterization of tight frames is using frame potential.
\begin{definition}
	Let $X=\{x_i\}_{i=1}^m$ be a collection of vectors in $\RR^n$. The frame potential for $X$ is the quantity
	\[FP(X)=\sum_{i=1}^{m}\sum_{j=1}^{m}|\langle x_i, x_j\rangle|^2.\]
\end{definition}
\begin{theorem}[\cite{BF}]\label{potential}
	Let $m\geq n$. If $X=\{x_i\}_{i=1}^m$ is any set of unit norm vectors in $\RR^n$, then
	\[FP(X)\geq \frac{m^2}{n}\] with equality if and only if $X$ is a tight frame.
\end{theorem}

If $X$ is a spherical two-distance set in $\RR^n$ and is also a tight frame for $\RR^n$, then we call $X$ a {\it two-distance tight frame}. Moreover, if in addition the angle set of $X$ is $\{\alpha, -\alpha\}$, then $X$ is called an {\it equiangular tight frame} or an ETF for short. 

An immediate consequence of Theorem \ref{potential} for the case of regular two-distance sets is as follows.
\begin{corollary}\label{co1}
	Let $X$ be a regular two-distance set of $m$ vectors in $\RR^n$ at angles $\alpha, \beta$ with respective multiplicities $k_\alpha, k_\beta$. Then
	\[1+k_\alpha \alpha^2+k_\beta\beta^2\geq \frac{m}{n},\]
	with equality if and only if $X$ is a two-distance tight frame.
\end{corollary}

\begin{definition}
	Given two frames $X=\{x_i\}_{i=1}^m$ and $Y=\{y_i\}_{i=1}^m$ for $\RR^n$. $X$ and $Y$ are said to be unitarily equivalent if there exists an unitary operator $U$ on $\RR^n$ such that $y_i=Ux_i$, for all $i\in [m]$. 
\end{definition}
It is known that two frames $X$ and $Y$ are unitarily equivalent if and only if their Gram matrices are equal.

Follow by the book \cite{W}, we now define a balanced set of vectors.
\begin{definition}
	A set of vectors $\{x_i\}_{i=1}^m$ in $\RR^n$ is said to be balanced if $\sum_{i=1}^mx_i=0.$
\end{definition}

A simple characterization of balanced sets is as follows.
\begin{proposition}\label{pro1}
	A set $X=\{x_i\}_{i=1}^m$ is balanced if and only if each row of its Gram matrix sums to zero.
\end{proposition}
\begin{proof}
	Suppose that for each $i$, $\sum_{j=1}^{m}\langle x_i, x_j\rangle=0$. Then we have
	\[\|\sum_{i=1}^{m}x_i\|^2=\langle \sum_{i=1}^{m}x_i, \sum_{j=1}^{m}x_j\rangle=\sum_{i,j=1}^{m}\langle x_i, x_j\rangle=0.\]
	So $X$ is balanced.
	The converse is obvious. 
\end{proof}

Given a dimension $n$, by maximal spherical two-distance sets (similarly maximal ETFs, maximal equiangualar lines) we mean the largest cardinality of such sets which can exist in $\RR^n$. 

The outline of this paper is as follows. In section 2, we present various properties of regular two-distance sets, in particular two-distance tight frames. In Section 3, we will then give several constructions of such sets, focusing on constructing regular two-distance frames with large cardinality. We conclude in Section 4 by discussing a connection between spherical two-distance sets and equiangular lines.  Several examples of the existence/non-existence of maximal equiangular lines and quasi-symmetric designs are also given. 
\section{Properties of regular two-distance sets}
In this section, we will present some properties of regular two-distance sets. In particular, we give sharp upper bounds on the maximum size of regular two-distance sets when both angles are positive or negative. We also show that if a regular two-distance set has a large cardinality, then it must be balanced.  Various properties for the special case where two-distance sets form tight frames for the space are also discussed.

We have mentioned that an upper bound for the maximum size of spherical two-distance sets of distinct vectors in $\RR^n$ is $\frac{n(n+3)}{2}$. If $n>2$, then this bound can be achieved only if $n=(2k+1)^2-3$ for $k\in \mathbb{N}$. For other dimensions, the bound is reduced to $\frac{n(n+1)}{2}$ with an exception for the case $n=5$, where the maximum size is $16$. We will now see that these upper bounds can be improved if both angles are either positive or negative.

We first consider the case where both angles are negative. Although the result can be deduced from the Rankin bound on the maximum number of spherical caps, see \cite{R}, also in \cite{CHS}, we will give a direct proof for this case below. Actually, the following theorem will give a sharp upper bound for cardinalities of sets that have negative/non-positive inner products between the vectors. 
\begin{theorem}\label{nega angles}
		Let $\{x_i\}_{i=1}^m$ be any set of non-zero vectors in $\RR^n$.
		\begin{enumerate}
			\item If $\langle x_i,x_j\rangle <0$  for all $i\not= j$,
			then $m\le n+1$.
			\item If $\langle x_i,x_j\rangle \leq0$  for all $i\not= j$,
			then $m\le 2n$.
		\end{enumerate} 
	
\end{theorem}

	\begin{proof} (1): 
		Let us do this by induction on dimension $n$.  For $n=2$ the largest set of vectors with negative angles is 3.  So assume the result is true for $n$ and consider $n+1$.  
		
		Without loss of generality, we can assume that $x_1$ is unit norm. Let $P$ be the orthogonal projection onto $\spn\{x_1\}$. Then for all $x\in \mathbb{R}^{n+1}$, 
		\[Px=\langle x, x_1\rangle x_1.\] Let consider the set of vectors $\{(I-P)x_i\}_{i=2}^m$. Clearly, this set lies on the hyperplane $x_1^\perp$. Moreover, for any $i\not=j$, we have
		\[\langle (I-P)x_i, (I-P)x_j\rangle=\langle x_i, x_j\rangle-\langle Px_i, Px_j\rangle.\]
		Since 
		\[\langle Px_i, Px_j\rangle=\langle\langle x_i, x_1\rangle x_1, \langle x_j, x_1\rangle x_1\rangle=\langle x_i, x_1\rangle \langle x_j, x_1\rangle>0,\]
		it follows that 
		\[\langle (I-P)x_i, (I-P)x_j\rangle<0 \mbox{ for all } i\not=j.\]
		By the induction hypothesis we must have $m-1\leq n+1$. So $m\leq n+2$.
		
		(2): We observe that the for $n=2$, the largest set of non-zero vectors with non-positive inner products is 4. Repeating the proof as in (1), noting that the set $\{(I-P)x_i\}_{i=2}^m$ contains at most one zero vector, we get the desired claim.
	\end{proof}

Now we will give a method to construct balanced, regular two-distance sets in one lower dimension from non-balanced ones. As consequences, we will get conditions for regular two-distance sets to be balanced as well as the upper bounds on the maximum size of the sets where both angles are non-negative.  

 \begin{theorem}\label{thm2}
 	Let $X=\{x_i\}_{i=1}^m$ be a regular two-distance set of distinct vectors in  $\mathbb{R}^n$ with its Grammian constant $c$. Let $\alpha$ and $\beta$ be its angles with multiplicities $k_\alpha$ and $k_\beta$, respectively. Assume that $X$ is not balanced and let $P$ be the orthogonal projection onto $\spn\{z\}$, where $z=\sum_{i=1}^{m}x_i$. Then $Y=\left\{\frac{(I-P)x_i}{\|(I-P)x_i\|}\right\}_{i=1}^m$ is a regular two-distance set of distinct vectors in $\mathbb{R}^{n-1}$ at angles $\frac{m}{m-c}\left(\alpha-\frac{c}{m}\right)$ and $\frac{m}{m-c}\left(\beta-\frac{c}{m}\right)$ with respective multiplicities $k_\alpha$ and $k_\beta$. Moreover, this set is balanced.
 \end{theorem}
 \begin{proof}
 For every $x\in \mathbb{R}^n,$ we have that
 	\[Px=\left\langle x, \dfrac{z}{\|z\|}\right\rangle\dfrac{z}{\|z\|}.\]

 	Now we compute
 	\[\|z\|^2=\langle \sum_{i=1}^{m}x_i, \sum_{j=1}^{m}x_j\rangle=mc,\]
 	and for all $i$,
 	\[\|(I-P)x_i\|^2=\|x_i\|^2-\|Px_i\|^2=1-\dfrac{1}{\|z\|^2}|\langle x_i, z\rangle|^2=1-\dfrac{c^2}{mc}=\dfrac{m-c}{m}.\]
 	Set $y_i=\frac{(I-P)x_i}{\|(I-P)x_i\|}$, we have
 	\begin{align*}
 	\langle y_i, y_j\rangle&=\dfrac{m}{m-c}\left(\langle x_i, x_j\rangle-\langle Px_i, Px_j\rangle\right)\\
 	&=\dfrac{m}{m-c}\left(\langle x_i, x_j\rangle-\frac{1}{\|z\|^2}\langle x_i, z\rangle\langle x_j, z\rangle\right)\\
 	&=\dfrac{m}{m-c}\left(\langle x_i, x_j\rangle-\frac{c^2}{\|z\|^2}\right)\\
 	&=\dfrac{m}{m-c}\left(\langle x_i, x_j\rangle-\frac{c}{m}\right).\\
 	\end{align*}
 	This implies that $Y$ is a two-distance set. Since $X$ is regular, it follows that $Y$ is regular with the same multiplicities as of $X$.
 	The vectors $y_i$'s are distinct since 	
 	\[\dfrac{m}{m-c}\left(\langle x_i, x_j\rangle-\frac{c}{m}\right)=1 \mbox{ if and only if } \langle x_i, x_j\rangle=1.\]
 	To show that $Y$ is balanced, we compute its Grammian constant. For any $i$, we have
 	\[\sum_{j=1}^{m}\langle y_i, y_j\rangle=\frac{m}{m-c}\left(\sum_{j=1}^{m}\langle x_i, x_j\rangle-c\right)=0,\] which is the claim.
 \end{proof}

  The orthogonal projection of a frame is also a frame for the range space with the same bounds. Hence, the following result is obvious.
  \begin{corollary}
  	If $X=\{x_i\}_{i=1}^m$ is a regular, two-distance tight frame for $\RR^n$ such that $X$ is not balanced, then the set $Y=\{y_i\}_{i=1}^m$ constructed in Theorem \ref{thm2} is a balanced, regular, two-distance tight frame for $\mathbb{R}^{n-1}$. 
  \end{corollary}
  \begin{theorem}\label{thm3}
  	Let $X=\{x_i\}_{i=1}^m$ be a regular two-distance set of distinct vectors in $\mathbb{R}^n, n\geq 7$. Then $X$ is balanced if $m> \frac{(n-1)n}{2}$ and $n\not=(2k+1)^2-2$, $k\in \mathbb{N}$. For the case $n=(2k+1)^2-2$, for some $k\in \mathbb{N}$, we need the condition $m>\frac{(n-1)(n+2)}{2}$ in order for $X$ to be balanced.  Moreover, these sets must have one positive angle and one negative angle.
  \end{theorem}

  \begin{proof}
  	We have seen that the upper bound for the number of vectors of any two-distance set in $\RR^n$ is $m\leq\frac{n(n+3)}{2}$ by the Harmonic bound \eqref{eneq1}, and if $n\not=(2k+1)^2-3$, $k\in \mathbb{N}$, then $m\leq\frac{n(n+1)}{2}$ by \eqref{eneq2}. Therefore, if $X$ is not balanced and the number of vectors $m$ satisfies the condition of the theorem, then by Theorem \ref{thm2}, we can construct another two-distance set of $m$ distinct vectors in $\mathbb{R}^{n-1}$. But this set has the number of vectors greater than the upper bound above, which cannot happen.
  	
  	For the ``moreover'' part, we observe that if $X$ has two non-negative angles, then it cannot be balanced. Note also that by Theorem \ref{nega angles}, $X$ cannot have both non-positive angles. This completes the proof.
  \end{proof}
  
  Theorem \ref{thm3} gives the upper bound for the cardinalities of non-balanced, regular two-distance sets, in particular for such sets with two non-negative angles.
  \begin{corollary}\label{posi angles}
  	Let $X=\{x_i\}_{i=1}^m$ be a regular two-distance set of distinct vectors in $\mathbb{R}^n$, $n\geq 7$ with both non-negative angles. Then we have the following:
  	\begin{enumerate}
  		\item If $n\not=(2k+1)^2-2$ for all $k\in \mathbb{N}$, then $m\leq \frac{(n-1)n}{2}$.
  		\item If $n=(2k+1)^2-2$ for some $k\in \mathbb{N}$, then $m\leq \frac{(n-1)(n+2)}{2}$.
  	\end{enumerate}
  \end{corollary}
  Note that the conditions on the number of vectors for the sets to be balanced in Theorem \ref{thm3} cannot not be lowered. In other words, the bounds on Corollary \ref{posi angles} are sharp. We will see these by examples in Section 3. Later, we also see that the properties for angles in Theorem \ref{thm3} hold true for two-distance tight frames of any size.

  We now interested in the case where two-distance sets form frames for the space, especially tight frames. The following result shows that we can get frames if the cardinalities of two-distance sets are large.

 \begin{proposition}\label{frame}
 	Let $X$ be a two-distance set of $m$ distinct vectors in $\RR^n, n\geq 7$. Then $X$ is a frame for $\RR^n$ if one of the following conditions hold:
 	\begin{enumerate}
 		\item $m>\frac{(n-1)n}{2}$ and $n\not=(2k+1)^2-2$ for all $k\in \mathbb{N}$.
 		\item $m>\frac{(n-1)(n+2)}{2}$ and $n=(2k+1)^2-2$ for some $k\in \mathbb{N}$.
 	\end{enumerate}
 \end{proposition}
 
  \begin{proof}
  	We will give a proof for (1). A proof for (2) is similar. Suppose by way of  contradiction that $X$ is not a frame for $\RR^n$. Hence $X$ does not span $\RR^n$. Therefore $X$ is a two-distance set for a subspace of dimension at most $n-1$. Since $n-1\not=(2k+1)^2-3$, it follows that the cardinality of $X$ is at most $\frac{(n-1)n}{2}$, which cannot happen by condition (1).
  \end{proof}
  \begin{remark}
  	 	\begin{enumerate}   			
		\item With the same arguments, it is simple to get similar results as in Theorem \ref{thm3}, Corollary \ref{posi angles}, and Proposition \ref{frame} for dimensions less than 7.
  		\item We should point out that a similar result to Proposition \ref{frame} for maximal equiangular lines is not true, i.e., given a set of maximal equiangular lines in $\RR^n$, let $X$ be a collection of vectors spanning each line, then $X$ may not span the space. A simple counterexample is that the maximal number of equiangular lines in $\RR^4$ is 6 and we can use the 6 lines in the subspace $\RR^3$ to get them.
  	\end{enumerate}	
  \end{remark}
  
  Tight frames has been shown to be very useful for many applications since they have both redundant and basis-like properties. In the language of design theory, a balanced tight frame is call a 2-design, see \cite{BGOY, DJS}.
  In the following, we will give a characterization of two-distance tight frames. As a consequence, every two-distance tight frame with angles $\alpha\not=-\beta$ is always regular.
  \begin{theorem}\label{charac tight} 
  	Let $X=\{x_i\}_{i=1}^m$ be a two-distance frame at angles $\alpha$ and $\beta$.  
  	The following are equivalent:
  	\begin{enumerate}
  		\item $X$ is a $m/n$-tight frame.
  		\item For some $\mathcal{J}\subset [m]$ with $\spn\{x_i\}_{i\in \mathcal{J}}=\RR^n$,
  		\[ \alpha\sum_{j\in \mathcal{I}^\alpha_i}x_j+\beta\sum_{j\in \mathcal{I}^\beta_i}x_j= \left(\frac{m}{n}-1\right)x_i, \mbox{ for all }i\in \mathcal{J}.\]
  	\end{enumerate}
  \end{theorem}
  
  \begin{proof}
  	$(1) \Rightarrow (2)$:  Since the frame is $m/n$-tight, for any $i\in [m]$, we have
  	\begin{align*}
  	\frac{m}{n}x_i &= \sum_{j=1}^m\langle x_i,x_j\rangle x_j\\
  	&= \sum_{j\in \mathcal{I}^\alpha_i}\langle x_i,x_j\rangle x_j+\sum_{j\in \mathcal{I}^\beta_i}\langle x_i,x_j\rangle x_j+\langle x_i,x_i\rangle x_i\\
  	&= \alpha\sum_{j\in \mathcal{I}^\alpha_i}x_j+\beta\sum_{j\in \mathcal{I}^\beta_i}x_j+x_i, 
  	\end{align*}
  	so (2) follows.

  	$(2)\Rightarrow (1)$:  (2) implies that if the frame operator of $X$ is $S$ then $Sx_i=\frac{m}{n}x_i$ for all $i\in \mathcal{J}$. Since $\spn\{x_i\}_{i\in \mathcal{J}}=\RR^n$, it follows that $Sx=\frac{m}{n}x$ for all $x\in \RR^n$. 
  \end{proof}

  \begin{proposition}\label{pro2}
  	If $X=\{x_i\}_{i=1}^m$ is a two-distance tight frame for $\RR^n$ at angles $\alpha, \beta$ and $\alpha\not=-\beta$, then $X$ is regular. Moreover, the Grammian constant of $X$ is either $0$ or $m/n$.
  \end{proposition}
  \begin{proof}
  	By (2) of Theorem \ref{charac tight}, for each $i\in [m]$, we have 
  	\[\alpha\sum_{j\in \mathcal{I}^\alpha_i}x_j+\beta\sum_{j\in \mathcal{I}^\beta_i}x_j= \left(\frac{m}{n}-1\right)x_i.\]
  	Taking the inner product both sides  of this equation with $x_i$, we get
  	\[|\mathcal{I}^\alpha_i|\alpha^2+(m-|\mathcal{I}^\alpha_i|-1)\beta^2=\frac{m}{n}-1.\]
  	Solving for $|\mathcal{I}^\alpha_i|$ we have
  	\[|\mathcal{I}^\alpha_i|=\dfrac{\frac{m}{n}-1+(1-m)\beta^2}{\alpha^2-\beta^2}
  	,\] which is independent of $i$. So $X$ is regular. 
  	
  	For the ``moreover'' part, note that the row sum of the Gram matrix $G$ of $X$ is an eigenvalue of $G$, the conclusion hence follows.
  \end{proof}
  \begin{remark}
  	For the case the angles $\alpha=-\beta$, i.e., for the case ETFs, the paper \cite{FJMPW} showed that the frames still might be regular. One construction of regular ETFs in \cite{FJMPW} is using Steiner systems and real Hadamard matrices. This paper also mentioned that there is no non-balance regular ETF of 28 vectors in $\RR^7$ because of the non-existence of the corresponding strongly regular graph. By Theorem \ref{thm3}, we see that actually, there is no non-balanced, regular ETFs of $\frac{n(n+1)}{2}$ vectors in $\RR^n$, for all $n$. In other words, all such regular maximal ETFs, if they exist, must be balanced.
  \end{remark}


  We have seen that a necessary condition for a regular two-distance frame of $m$ vectors in $\RR^n$ to be tight is that its Grammian constant must be either $0$ or $m/n$. However, the following examples will show that this is not a sufficient condition.
  
  \begin{example}\label{ex1}
  	Let $a, b$ be two numbers such that $a^2+b^2=1$ and $a^2\not=2b^2$. We consider a frame in $\RR^3$ with the vectors of the form:
  	\[x_1=(a, -b, 0); \quad x_2=(0, b, a); \quad x_3=(-a, -b, 0);\quad x_4=(0, b, -a).\]
  	It is simple to check that this frame has two angles $b^2-a^2$ and $-b^2$. Moreover, $\sum_{i=1}^{4}x_i=0$. The condition $a^2\not=2b^2$ implies that it cannot be tight.	
  \end{example}
  Now we will give an example of a regular two-distance frame of $4$ vectors in $\RR^4$ with its Grammian constant $4/4$ but it is not tight.

  \begin{example}
  
  Let the frame to be
  		\[
  	x_1=\left(\frac{\sqrt{3}}{4}, -\frac{3}{4}, 0, \frac{1}{2}\right); \quad 
  	x_2=\left(0, \frac{3}{4},\frac{\sqrt{3}}{4}, \frac{1}{2}\right);\]
  	\[x_3=\left(-\frac{\sqrt{3}}{4},-\frac{3}{4},0,\frac{1}{2}\right);\quad
  	x_4=\left(0,\frac{3}{4},-\frac{\sqrt{3}}{4},\frac{1}{2}\right).
  	\]
  	Then $\{x_i\}_{i=1}^4$ is a two-distance frame for $\RR^4$ at angles $\alpha=-5/16$ and $\beta=5/8$ with respective multiplicities $k_\alpha=2$ and $k_\beta=1$.
  	We can check that this frame is not tight and 
  	\[1+k_\alpha\alpha+k_\beta\beta=1=4/4.\]
 \end{example}

  It is well-known that the Naimark complement of an equiangular tight frame is also an equiangular tight frame. This is also the case for two-distance tight frames. 
  
  \begin{theorem}\label{Naimark}
  	If $\{x_i\}_{i=1}^m$ is a regular two-distance tight frame for $\RR^n$ at angles $\alpha, \beta$ with multiplicities $k_\alpha, k_\beta$, respectively, then its Naimark complement is also a two-distance tight frame for $\RR^{m-n}$ at angles $-\frac{n}{m-n}\alpha, -\frac{n}{m-n}\beta$ with respective multiplicities $k_\alpha$ and $k_\beta$.
  \end{theorem}
  \begin{proof}
  	Set $u_i=\sqrt{n/m}x_i$ then $\{ u_i \}_{i=1}^m$ is a Parseval frame for $\RR^n$. Let $\{v_i\}_{i=1}^m$ be its Naimark complement. Note that $\{(u_i, v_i)\}_{i=1}^{m}$ is an orthonormal basis for $\RR^{m}$ and $\{v_i\}_{i=1}^m$ is a Parseval frame for $\RR^{m-n}$. Moreover, 
  	$$\Vert v_i\Vert^2=1-\Vert u_i\Vert^2=1-\frac{n}{m}.$$
  	Set $y_i=\sqrt{\dfrac{m}{m-n}}v_i$ then $\{ y_i \}_{i=1}^m$ is a unit norm tight frame for $\RR^{m-n}$.
  	
  	We have 
  	\begin{align*}
  	\langle y_i, y_j\rangle&=\frac{m}{m-n}\langle v_i, v_j\rangle\\
  	&=-\frac{m}{m-n}\langle u_i, u_j\rangle\\
  	&=-\frac{m}{m-n}.\dfrac{n}{m}\langle x_i, x_j\rangle=-\frac{n}{m-n}\langle x_i, x_j\rangle.
  	\end{align*}
  	This completes the proof
  \end{proof}
  
   As we have shown, regular two-distance sets of $m$ vectors in $\RR^n$ with Grammian constant $0$ or $m/n$ might not be tight frames. However, their angles have the same property as two-distance tight frames as in the following theorem. Recall that this property for angles is true for regular two-distance sets with large cardinalities by Theorem \ref{thm3}.
   
   \begin{theorem}\label{angles of tight}
   	Let $X=\{x_i\}_{i=1}^m$ be a regular two-distance set in $\RR^n$ at angles $\alpha, \beta$. Suppose that $m>n+1$ and the Grammian constant of $X$ is either $0$ or $m/n$. Then $\alpha\beta\leq 0$. In particular, any two-distance tight frame must have one non-negative angle and one non-positive angle.
   \end{theorem}
   \begin{proof}
   	Since $m>n+1$, by Theorem \ref{nega angles}, $\alpha$ and $\beta$ cannot be both negative. 
   	
   	If the Grammian constant $c=0$, then obviously the angles cannot be both non-negative. Now we consider the case $c=m/n$. Let $k_\alpha$ and $k_\beta$ be the multiplicities of $\alpha$ and $\beta$, respectively. Suppose that $\alpha, \beta>0$. Then by Corollary \ref{co1}, we have
   	\[\frac{m}{n}\leq 1+k_\alpha\alpha^2+k_\beta\beta^2< 1+k_\alpha\alpha+k_\beta\beta=\frac{m}{n},\] which cannot happen.  
   	
   	Now assume $X$ is a two-distance tight frame of $m$ vectors in $\RR^n$ at angles $\alpha$ and $\beta$. If $\alpha=-\beta$, then the conclusion is obvious. Otherwise, by Proposition \ref{pro2}, $X$ is regular and then the conclusion follows since its Grammian constant is always either $0$ or $m/n$. 
   \end{proof}
   
   \begin{remark}
   	\begin{enumerate}
   		\item The condition $m>n+1$ in Theorem \ref{angles of tight} is necessary. Example \ref{ex1} shows that there are regular two-distance sets of $4$ vectors in $\RR^3$ with both negative angles.
   		\item If $X$ is a two-distance tight frame of $n+1$ vectors in $\RR^n$ at angles $\alpha, \beta$, then its Naimark complement is a two-distance tight frame of $n+1$ vectors in $\RR$ at angles $-\frac{n}{m-n}\alpha, -\frac{n}{m-n}\beta$. Thus, we must have $|\frac{n}{m-n}\alpha|=|\frac{n}{m-n}\beta|$, and hence $\alpha=-\beta$. Actually, $X$ is obtained by negating some vectors of a simplex on $\RR^n$.
   	\end{enumerate}
   \end{remark}

  \begin{proposition}\label{zero_angle}
  	Let $X=\{x_i\}_{i=1}^m$ be a two-distance tight frame for $\RR^n$ at angles $\alpha, \beta$ with respective multiplicities $k_\alpha, k_\beta$. If $\alpha=0$, then we have one of the following:
  	\begin{enumerate}
  		\item $X$ is $(k_\beta+1)$ copies of an orthonormal basis of $\RR^n$.
  		\item The Naimark complement of $X$ is $(k_\beta+1)$ copies of an orthonormal basis of $\RR^{m-n}$.
  	\end{enumerate} 
  	
  \end{proposition}
  \begin{proof}
  	By assumption, $X$ must be regular. Moreover, we have that
  	\[1+k_\alpha \alpha+k_\beta\beta=0, \mbox{ or } 1+k_\alpha \alpha+k_\beta\beta=\frac{m}{n},\]
  	and
  	\[1+k_\alpha \alpha^2+k_\beta\beta^2=\frac{m}{n}.\]
  	Hence,  
  	\[k_\alpha\alpha(1-\alpha)+k_\beta\beta(1-\beta)\leq 0.\]

  	Therefore, if $\beta>0$ then $\beta=1$. This implies that $X$ is $(k_\beta+1)$ copies of an orthonormal basis of $\RR^n$.
  	
  	Consider the case $\beta<0$. Let $Y$ be the Naimark complement of $X$. Then by Theorem \ref{Naimark}, $Y$ has angles $a=-\frac{n}{m-n}\alpha$ and $b= -\frac{n}{m-n}\beta$ with respective multiplicities $k_\alpha$ and $k_\beta$. Hence, $a=0$ and $b>0.$
  	
  	But for the frame $Y$, we also have
  	\[k_\alpha a(1-a)+k_\beta b(1-b)\leq 0.\] 
  	This implies $b=1$ and $Y$ is $(k_\beta+1)$ copies of an orthonormal basis of $\RR^{m-n}$. 
  \end{proof}
 
  Before considering more properties of two-distance tight frames, we will present an interesting result about tight frames.

  		\begin{theorem}
  			Suppose $X=\{x_i\}_{i=1}^m$ is a Parseval frame with the property that its Gram matrix has constant row sum. Then either $X$ or its Naimark complement is balanced. More precisely, if $Y=\{y_i\}_{i=1}^m$ is its Naimark complement, then we have
  			\begin{enumerate}
  				\item $\sum_{i=1}^{m}x_i=0$ if and only if $\sum_{j=1}^{m}\langle y_i, y_j\rangle=1$ for all $i$.
  				\item  $\sum_{i=1}^{m}y_i=0$ if and only if $\sum_{j=1}^{m}\langle x_i, x_j\rangle=1$ for all $i$.
  			\end{enumerate}
  		\end{theorem}
  		\begin{proof}
  			Let $G$ be the Gram matrix of $X$. By assumption, $\sum_{j=1}^{m}G_{ij}=c$, for all $i$. 
  			 Then we have that
  			$ G\bold{1}=c\bold{1}.$	
  			Since $G$ has only 2 eigenvalues, $0$ and $1$, we have that $c=0$ or $c=1$.
  			If $c=0$ then $X$ is balanced by Proposition \ref{pro1}. 
  			
  			Consider the case $c=1$.
  			Note that $I-G$ is the Gram matrix of $Y$. Since
  			\[(I-G)\bold{1}=\bold{1}-G\bold{1}=0,\]
  			it follows that each row of $I-G$ sums to zero, which is the claim.
  			 			
  			Now we argue
  			\begin{align*}
  			\sum_{i=1}^{m}x_i=0 &\mbox{ iff } G\bold{1}=0
  			\mbox{ iff } (I-G)\bold{1}=\bold{1}
  			\mbox{ iff } \sum_{j=1}^{m}\langle y_i, y_j\rangle=1, \mbox{ for all i}, 
  			\end{align*}
  			which is part (1). With the same argument we will get part (2).
  		\end{proof}
 \begin{corollary}
 	Let $X=\{x_i\}_{i=1}^m$ be a regular, two-distance tight frame for $\RR^n$, and let $Y=\{y_i\}_{i=1}^m$ be its normalized Naimark complement. Then either $X$ or $Y$ is balanced.  Moreover, 
 	\begin{enumerate}
 		\item $\sum_{i=1}^{m}x_i=0$ if and only if $\sum_{j=1}^{m}\langle y_i, y_j\rangle=\frac{m}{m-n} \mbox{ for all i}$.\\
 		\item $\sum_{i=1}^{m}y_i=0$ if and only if $\sum_{j=1}^{m}\langle x_i, x_j\rangle=\frac{m}{n} \mbox{ for all i}.$ 
 	\end{enumerate}
 \end{corollary}

  For given $n$, it is known that there are only finitely many equiangular tight frames (up to unitary equivalence) in $\RR^n$. This is still the case for regular two-distance tight frames of distinct vectors.
  
  Indeed, suppose $X$ is a regular two-distance tight frame of $m$ distinct vectors in $\RR^n$ at angles $\alpha, \beta$, with respective multiplicities $k_\alpha, k_\beta$.
  	
  	Note that $(\alpha, \beta)$ is a solution of the system of equations
  	\begin{align}\label{eq3}
  	1+k_\alpha x+k_\beta y=0 \mbox{ and } 1+k_\alpha x^2+k_\beta y^2=m/n
  	\end{align}
  	or 
  	\begin{align}\label{eq4}
  		1+k_\alpha x+k_\beta y=m/n \mbox{ and } 1+k_\alpha x^2+k_\beta y^2=m/n,
  	\end{align}
  where the former corresponds to the case for which $X$ is balanced, and the laster is for the case where the Naimark complement of $X$ is balanced. It is easy to check that both systems \eqref{eq3} and \eqref{eq4} have two solutions. Moreover, for given $m$, there are at most $m-2$ possibilities for $k_\alpha$. Since $m\leq \frac{n(n+3)}{2}$ for any dimension $n$, it follows that there are finitely many two-distance tight frames for $\RR^n$. We will give a more precise result later in this section.
  
  If the number of vectors of a regular two-distance set is odd, then there is a restriction on multiplicities of its angles.
  
  \begin{proposition}\label{even multip}
  	Let $X$ be a regular two-distance set of $m$ vectors at angles $\alpha, \beta$, with respective multiplicities $k_\alpha, k_\beta$. If $m$ is odd, then both $k_\alpha$ and $k_\beta$ are even.
  \end{proposition}
  \begin{proof}
  	Let $G$ be the Gram matrix of $X$. This is a $m\times m$ self-adjoint matrix. Since $X$ is regular, each row of $G$ has exactly $k_\alpha$ elements $\alpha$ and $k_\beta$ elements $\beta$. It follows that both $mk_\alpha$ and $mk_\beta$ are even and so $k_\alpha$ and $k_\beta$ are even.
  \end{proof}
  
  We have mentioned that if $\alpha$ and $\beta$ are the angles of a regular two-distance tight frame of $m$ vectors in $\RR^n$, with multiplicities $k_\alpha$ and $k_\beta$, then $(\alpha, \beta)$ is one solution of either the system \eqref{eq3} or \eqref{eq4}. Let us denote by $(\alpha', \beta')$ the remaining solution of the system. Then a natural question is that whether $\alpha'$ and $\beta'$ are angles for some regular two-distance tight frame of $m$ vectors in $\RR^n$ with respective multiplicities $k_\alpha$ and $k_\beta$. In order to answer this question, we will first construct interesting matrices.
  \begin{theorem}\label{const matrix}
  	Let $X$ be a regular two-distance tight frame of $m$ vectors in $\RR^n$ at angles $\alpha, \beta$, and multiplicities $k_\alpha, k_\beta$, respectively. Let $G$ be its Gram matrix and $c$ be its Grammian constant. Let $G'$ be the matrix with all 1 in the diagonal and its off diagonal entries defined by 
  	\[G'_{ij}=\gamma-G_{ij},\]
  	where $\gamma=-\frac{2}{m-1}$ if $c=0$, and $\gamma=\frac{2(m-n)}{n(m-1)}$ if $c=m/n$.
  	
  	In other words, 
  	\[G'=(2-\gamma)I+\gamma J-G.\] 
  	Then $G'$ has the following properties.
  	\begin{enumerate}
  		\item $G'$ is self-adjoint and each row has exactly $k_\alpha$ elements $\gamma-\alpha$, and $k_\beta$ elements $\gamma-\beta$.
  		\item $G'$ has constant row sum. More precisely, 
  		\[1+k_\alpha(\gamma-\alpha)+k_\beta(\gamma-\beta) = 0\mbox{ if } \gamma=-\frac{2}{m-1}, \mbox{ and}\]
  		\[1+k_\alpha(\gamma-\alpha)+k_\beta(\gamma-\beta)=m/n \mbox{ if } \gamma=\frac{2(m-n)}{n(m-1)}.\]
  		\item We have
  		\[1+k_\alpha(\gamma-\alpha)^2+k_\beta(\gamma-\beta)^2=m/n.\]	
  	\end{enumerate}
  	
  \end{theorem}
  \begin{proof}
  	The claim (1) follows easily from the properties of $G$. For (2) and (3), we will prove the case $\gamma=-\frac{2}{m-1}$, since the other case is similar.
  	
  	For (2), note that by definition, $c=1+k_\alpha\alpha+k_\beta\beta=0$, so  
  	\begin{align*}
  	1+k_\alpha(\gamma-\alpha)+k_\beta(\gamma-\beta)&=1-(k_\alpha\alpha+k_\beta\beta)+(m-1)\gamma\\
  	&=2-(m-1)\frac{2}{m-1}\\
  	&=0,
  	\end{align*}
  	which is the claim.
  	For (3) we have
  	\begin{align*}
  	1+k_\alpha(\gamma-\alpha)^2+k_\beta(\gamma-\beta)^2&=1+k_\alpha\alpha^2+k_\beta \beta^2\\
  	&-2\gamma(k_\alpha\alpha+k_\beta\beta)+(m-1)\gamma^2.
  	\end{align*}
  	Using the fact that 
  	\[1+k_\alpha\alpha+k_\beta\beta=0 \mbox{ and } 1+k_\alpha\alpha^2+k_\beta \beta^2=m/n,\]
  	we will get the desired result.
  \end{proof}
  
  It turns out that for a fixed dimension $n$, there is only one case for which the solution $(\alpha', \beta')$ of the system \eqref{eq3} corresponds to angles of a two-distance tight frame. A similar result holds true when we consider the system \eqref{eq4}. The following lemma will play a role for showing this.
  
  \begin{lemma}\label{lem1}
  A $m\times m$ self-adjoint matrix $G$ is the Gram matrix of a two-distance tight frame of $m$ vectors for $\RR^n$ if and only if it satisfies the following conditions:
  \begin{enumerate}
  	\item $G^2=\frac{m}{n}G$.
  	\item $G_{ii}=1$ for all $i$.
  	\item There exist $\alpha$ and $\beta$ such that $G_{ij}$ equals either $\alpha$ or $\beta$, where $\alpha\not=\beta$.
  \end{enumerate}
  \end{lemma}
  \begin{proof}
  	If $G$ is the Gram matrix of a regular two-distance tight frame, then it is obvious that $G$ satisfies conditions $(1), (2)$ and $(3)$.
  	
  	Conversely, suppose $G$ satisfies $(1), (2)$ and $(3)$. Since $G^2=\frac{m}{n}G$, it follows that $0$ and $m/n$ are only two possible eigenvalues of $G$ . Thus, $G$ is positive semidefinite and hence it is the Gram matrix for some set of vectors. Since $\tr(G)=m>0$, $m/n$ must be an eigenvalue of $G$. Let $k$ be the multiplicity of $m/n$. Then $\tr(G)=m=k\frac{m}{n}$. This implies $k=n$ and so this set of vectors spans $\RR^n$, i.e., it is a frame for $\RR^n$. To be more precise, we can choose the vectors as in the following way. Let $D$ be the diagonal matrix of order $m$ of the form, 
  	$D=\begin{bmatrix}
  		\frac{m}{n}I&0\\
  		0&0
  	\end{bmatrix}$, where $I$ is the identity matrix of order $n$. Then there exists an unitary matrix $U$ of eigenvectors of $G$ such that
  	\[G=UDU^*=\begin{bmatrix}
  	U_1&U_2
  	\end{bmatrix}\begin{bmatrix}
  	\frac{m}{n}I&0\\
  	0&0
  	\end{bmatrix}\begin{bmatrix}
  	U_1^*\\
  	U_2^*
  	\end{bmatrix}=\frac{m}{n}U_1U_1^*,\]  where $U_1$ and $U_2$ are $m\times n$ and $m\times (m-n)$ submatrices of $U$ whose columns are eigenvectors of $G$ with eigenvalues $m/n$ and $0$, respectively. Now choose the set of vectors to be the columns of the $n\times m$ matrix $X=\sqrt{\frac{m}{n}}U_1^*$. These vectors form a two-distance tight frame since $XX^*=\frac{m}{n}U_1^*U_1=\frac{m}{n}I$ and its Gram matrix $G$ satisfies conditions $(2)$ and $(3)$. This completes the proof.
  \end{proof}

  	\begin{theorem}\label{thm4}
  		Let $X$ be a regular two-distance tight frame of $m$ vectors in $\RR^n$ at angles $\alpha$ and $\beta$, with multiplicities $k_\alpha$ and $k_\beta$. Let $G$ be its Gram matrix and $c$ be its Grammian constant. Let $G'$ be defined by
  		\[G'=(2-\gamma)I+\gamma J-G,\]
  		where $\gamma=-\frac{2}{m-1}$ if $c=0$, and $\gamma=\frac{2(m-n)}{n(m-1)}$ if $c=m/n$. 

We have the following:
\begin{enumerate}
	\item If $\gamma=-\frac{2}{m-1}$, then $G'$ is the Gram matrix of a regular two-distance tight frame $Y$ for $\RR^n$ if and only if $m=2n+1$.
	\item If $\gamma=\frac{2(m-n)}{n(m-1)}$, then $G'$ is the Gram matrix of a regular  two-distance tight frame $Y$ for $\RR^n$ if and only if $m=2n-1$.
\end{enumerate}  		
  	 Moreover, the angles of $Y$ are $\gamma-\alpha$ and $\gamma-\beta$ with the same multiplicities of $X$, i.e., $k_\alpha$ and $k_\beta$, respectively. $Y$ is balanced if $\gamma=-\frac{2}{m-1}$ and the Naimark complement of $Y$ is balanced if $\gamma=\frac{2(m-n)}{n(m-1)}$.
  	\end{theorem}
  	\begin{proof}
  		We will prove (1). The proof for (2) is similar. By Lemma \ref{lem1}, it is enough to find conditions for which  $G'^2=\frac{m}{n}G'$. 
  		
  		We have that
  		\begin{align*}
  		G'^2&=[(2-\gamma)I+\gamma J-G]^2\\
  		&=(2-\gamma)^2I+\gamma^2J^2+G^2+2\gamma(2-\gamma)J-2(2-\gamma)G-\gamma GJ-\gamma JG.
  		\end{align*}
  		Note that $GJ=JG=(1+k_\alpha\alpha+k_\beta\beta)J=0$, $G^2=\frac{m}{n}G$, and $J^2=mJ$. Hence,
  		\begin{align*}
  		G'^2&=(2-\gamma)^2I+\gamma[\gamma m+2(2-\gamma)]J-2(2-\gamma)G+\frac{m}{n}G\\
  		&=(2-\gamma)^2I+\gamma(2-\gamma)J-(2-\gamma)\left[2-\frac{m}{n(2-\gamma)}\right]G\\
  		&=(2-\gamma)\left[(2-\gamma)I+\gamma J-\left(2-\frac{m}{n(2-\gamma)}\right)G\right].
  		\end{align*}
  		Therefore $G'^2=\frac{m}{n}G'$ if and only if $2-\gamma=m/n$ if and only if $m=2n+1$.
  		
  		The remaining conclusions follow by Theorem \ref{const matrix}.
  	\end{proof}

  	
  	From Theorem \ref{thm4}, we see that for almost $n$ and $m$, there is at most one regular two-distance tight frame for $\RR^n$ with given multiplicities.
  	\begin{theorem}	For given $n, m, m>n+1$ and an integer $k$ in $[1, m-2]$.
  		\begin{enumerate}
  			\item 	If $m\not=2n+1$, then up to unitary equivalence and reordering the frame vectors, there is at most one balanced, regular two-distance tight frame of $m$ vectors for $\RR^n$ with multiplicities $k$ and $m-k-1$.
  			\item Similarly, if $m\not=2n-1$, then up to unitary equivalence and reordering the frame vectors, there is at most one non-balanced, regular two-distance tight frame of $m$ vectors for $\RR^n$ with multiplicities $k$ and $m-k-1$.
  		\end{enumerate} 
  	\end{theorem}
  	\begin{proof}
  			(1): We proceed by way of contradiction. Suppose $X$ and $Y$ are frames satisfying (1). Denote by $\alpha, \beta$ the angles of $X$, and $\alpha', \beta'$ the angles of $Y$. Suppose $\alpha$ and $\alpha'$ have the same multiplicities $k$. Since $X, Y$ are balanced and tight, it follows that $(\alpha, \beta)$ and $(\alpha', \beta')$ are solutions of the system of equations:
  			\[1+kx+(m-k-1)y=0 \mbox{ and } 1+kx^2+(m-k-1)y^2=\frac{m}{n}.\]
  			We can check that the point $\left(-\frac{1}{m-1},-\frac{1}{m-1}\right)$ lies on the line $1+kx+(m-k-1)y=0$ for arbitrary $k$ and this system always has two solutions. Moreover, if $(\alpha, \beta)$ is a solution of the system then $(\alpha',\beta')=\left(-\frac{2}{m-1}-\alpha,-\frac{2}{m-1}-\beta\right)$ is the remaining solution. Thus, if $G$ is the Gram matrix of $X$, then $G'=(2-\gamma)I+\gamma J-G$ is the Gram matrix of $Y$, where $\gamma=-\frac{2}{m-1}$.
  			By Theorem \ref{thm4}, we must have $m=2n+1$, which is a contradiction. So $X, Y$ cannot be both balanced, two-distance tight frames.

  				(2): In this case, $(\alpha, \beta)$ and $(\alpha', \beta')$ must be solutions of the system of equations:
  				\[1+kx+(m-k-1)y=\frac{m}{n} \mbox{ and } 1+kx^2+(m-k-1)y^2=\frac{m}{n}.\]
  				This system always has 2 solutions, and
  				\[\alpha+\alpha'=\frac{2(m-n)}{n(m-1)}, \mbox{ and }\beta+\beta'=\frac{2(m-n)}{n(m-1)}.\]
  			
  				Again, by Theorem \ref{thm4}, we have $m=2n+1$, which contradicts our assumption. This completes the proof.
  	\end{proof}

  	\begin{remark}
  		Given $n, m$ such that $m\not= 2n+1$ and $k\in [1, m-2]$. It is possible for the existence of both balanced and non-balanced, regular two-distance tight frames of $m$ vectors in $\RR^n$ with multiplicities $k$ and $m-k-1$. For instance, let $X$ be 2 copies of an orthonormal basis for $\RR^n$ and $Y$ be its Naimark complement.
  	\end{remark}
  	
\section{Construction regular two-distance sets}
Recall that Theorem \ref{thm2} gave a method for constructing a regular two-distance set from a given one. In this section, we will continue to present some other constructions of regular two-distance sets, in particular two-distance tight frames. We focus on constructing these sets with large cardinality. These constructions include one family of maximal two-distance sets that has been constructed in some previous paper, for example in \cite{BGOY}. One of the main tools we use here is from combinatorial designs. Combinatoric configurations has been used vastly in frame theory. For instance, it is well-known that equiangular tight frames can be constructed from difference sets or Steiner systems, see \cite{FMT, XZG}. Likewise, divisible difference sets and partial difference sets are used to construct biangular tight frames \cite{CFHT}. We will continue to exploit some families of block designs to construct desired sets.

	\begin{definition}
		
		A $t$-$(v, k, \lambda)$ block design, or a $t$-design for short, is a pair $(V, \mathcal{B})$ where $V$ is a $v$-set of points and $\mathcal{B}$ is a collection of $k$-subsets of $V$ (blocks) with the property that every $t$-subset of $V$ is contained in exactly $\lambda$ blocks. 
		
		 If $t=2$, then the design is called a balanced incomplete block design or BIBD. 
	\end{definition}
	
	Given a $2$-$(v, k, \lambda)$ design, each element of $V$ is contained
	in exactly $r$ blocks. It is customary to denote by $b$ the number of blocks. The
	numbers $v, b, r, k$, and $\lambda$ are parameters of the BIBD.

	The following proposition gives a few simple facts about block designs, see \cite{CD}.
	\begin{proposition}
		For any $t$-$(v, k, \lambda)$ block design, the following conditions hold:
		\begin{enumerate}
			\item $vr=bk$,
			\item $r(k-1)=\lambda(v-1)$ if $t=2$.
		\end{enumerate}
		Furthermore, a $t$-block design is also a $(t-1)$-block design for $t>1$.
	\end{proposition}

	A block design can be represented by a matrix called the {\it incidence matrix}. The incidence matrix of a BIBD $(V, \mathcal{B})$ with parameters $v, b, r, k, \lambda$ is a $v\times b$ matrix $A = (a_{ij})$, in which $a_{ij} = 1$ when the ith element of $V$ occurs in the jth block of $\mathcal{B}$, and $a_{ij} = 0$ otherwise.
	
	Before constructing regular two-distance sets, we will use BIBDs to construct some balanced frames and present some simple properties.   
	\begin{proposition}\label{pro3}
		Let $X=\{x_i\}_{i=1}^m$ be a balanced frame for $\RR^n$ with bounds $A$ and  $B$. Let $([m],\mathcal{B})$ be a $(m, k, \lambda)$ BIBD. For each block $\mathcal{J}\in \mathcal{B}$, define $y_{\mathcal{\mathcal{J}}}=\sum_{i\in \mathcal{J}}x_i$. Then $Y=\{y_\mathcal{J}\}_{\mathcal{J}\in \mathcal{B}}$ is a frame for $\RR^n$ with bounds $(r-\lambda)A$ and $(r-\lambda)B$. Moreover, $Y$ is balanced. 
	\end{proposition}
	\begin{proof}
		For any $x$ we have
		\begin{align*}
		\sum_{\mathcal{J}\in \mathcal{B}}|\langle x, y_\mathcal{J}\rangle|^2=\sum_{\mathcal{J}\in \mathcal{B}}\left(\sum_{i\in \mathcal{J}}|\langle x, x_i\rangle|^2+\sum_{i, j\in \mathcal{J}, i\not=j}\langle x, x_i\rangle\langle x, x_j\rangle\right).
		\end{align*}
		Since every element $x_i$ is contained in $r$ blocks and every 2-subset $\{x_i, x_j\}$ is contained in $\lambda$ blocks, it follows that 
		\begin{align*}
		\sum_{\mathcal{J}\in \mathcal{B}}|\langle x, y_{\mathcal{J}}\rangle|^2&=r\sum_{i=1}^{m}|\langle x, x_i\rangle|^2+\lambda\sum_{i, j=1, i\not=j}^m\langle x, x_i\rangle\langle x, x_j\rangle\\
		&=r\sum_{i=1}^{m}|\langle x, x_i\rangle|^2+\lambda\left[\left(\sum_{i=1}^{m}\langle x, x_i\rangle\right)^2-\sum_{i=1}^{m}|\langle x, x_i\rangle|^2\right]\\
		&=(r-\lambda)\sum_{i=1}^{m}|\langle x, x_i\rangle|^2+\lambda\left\langle x, \sum_{i=1}^{m}x_i\right\rangle^2\\
		&=(r-\lambda)\sum_{i=1}^{m}|\langle x, x_i\rangle|^2.
		\end{align*}
		The conclusion follows. The ``moreover" part is clear since
		$\sum_{\mathcal{J}\in \mathcal{B}}y_\mathcal{J}=r\sum_{i=1}^{m}x_i.$
	\end{proof}
	\begin{proposition}
		Let $S$ be the frame operator for a balanced frame $\{x_i\}_{i=1}^m$ and $S'$ be the frame operator for the frame $\{y_\mathcal{J}\}_{\mathcal{J}\in \mathcal{B}}$ constructed as in Proposition \ref{pro3}. Then $S'=(r-\lambda)S$.
	\end{proposition}
	\begin{proof}
		By Proposition \ref{pro3}, for any $x$ we have that
		\begin{align*}
		\langle S'x, x\rangle=\sum_{\mathcal{J}\in \mathcal{B}}|\langle x, y_{\mathcal{J}}\rangle|^2&=(r-\lambda)\sum_{i=1}^m|\langle x, x_i\rangle|^2\\
		&=(r-\lambda)\langle Sx, x\rangle=\langle(r-\lambda)Sx,x\rangle.
		\end{align*}
		This implies the desired claim.

	\end{proof}
	\begin{corollary}
		If $\{x_i\}_{i=1}^m$ is a balanced $A$-tight frame, then $\{y_\mathcal{J}\}_{\mathcal{J}\in \mathcal{B}}$ is a balanced $(r-\lambda)A$-tight frame.
	\end{corollary}

A special type of $t$-designs is the so called quasi-symmetric designs. These designs have the property that the cardinality of the intersection of any two blocks of the designs are either $x$ or $y$. They have been studied extensively due to their connections with strongly regular graphs.

	\begin{definition}
		A $t$-$(v, k, \lambda)$ block design $(V, \mathcal{B})$ is quasi-symmetric with intersection numbers $x$ and $y$ if any two blocks of $\mathcal{B}$ intersect in either $x$ or $y$ points.
	\end{definition}
	\begin{definition}
		The block graph $\Gamma$ of quasi-symmetric $2$-$(v, k, \lambda)$ design $(V, \mathcal{B})$ with intersection numbers $x$ and $y$ $(x<y)$ is the graph with vertex set being the blocks of $\mathcal{B}$, and where blocks $\mathcal{J}$ and $\mathcal{J}'$ are adjacent if and only if
		$|\mathcal{J} \cap \mathcal{J}'|=y$.
	\end{definition}
	
	It is well known that the block graph of a quasi-symmetric
	design is a strongly regular graph \cite{IS}. 
	
	\begin{theorem} \label{SRG}
		Let $(V, \mathcal{B})$ be a $2$-$(v, b, r, k, \lambda)$ quasi-symmetric design with intersection numbers $x$ and $y$, where $x<y$. Then $\Gamma$ is a strongly regular graph $(n, s, \mu_1, \mu_2)$. The parameters of $\Gamma$ are given by 
		\[n=b,\quad s=\dfrac{k(r-1)-x(b-1)}{y-x},\quad \mu_1=s+\theta_1+\theta_2+\theta_1\theta_2,  \quad \mu_2=s+\theta_1\theta_2,\] where
		\[\theta_1=\frac{r-\lambda-k+x}{y-x}, \quad
		\theta_2=\frac{x-k}{y-x}.\]
	\end{theorem}

	Thus, given a 2-$(v, b, r, k, \lambda)$ quasi-symmetric design $(V, \mathcal{B})$ with intersection numbers $x$ and $y$, where $x < y$. Then for each block $\mathcal{J}$ of $\mathcal{B}$,
	there are exactly $s$ blocks of $\mathcal{B}$ for which each of them intersects $\mathcal{J}$ at $y$ points. This result ensures the regularity of two-distance frames in the following construction.

	\begin{theorem}\label{construction1}
		Let $([n], \mathcal{B})$ be a $2$-$(n, b, r, k, \lambda)$ quasi-symmetric  design with intersection numbers $x$ and $y$. Let $s$ be the number defined as in Theorem \ref{SRG}. Then we have the following constructions of regular two-distance frames. 
		\begin{enumerate}
			\item Let $\{e_i\}_{i=1}^n$ be the standard orthonormal basis for $\RR^n$. For each block $\mathcal{J}\in \mathcal{B}$, define a vector $x_{\mathcal{J}}$ to be 
			\[x_{\mathcal{J}}=\dfrac{1}{\sqrt{k}}\sum_{i\in \mathcal{J}}e_i.\]
			Then the family $\{x_{\mathcal{J}}\}_{\mathcal{J}\in \mathcal{B}}$ forms a regular two-distance frame for  $\RR^n$ at angles $x/k$ and $y/k$ with multiplicities $b-s-1$ and $s$, respectively. This frame is not tight.
			\item Similarly, let $\{\varphi_i\}_{i=1}^{n}$ be a simplex for $\RR^{n-1}$. For each block $\mathcal{J}\in \mathcal{B}$, define a vector $x_\mathcal{J}$ to be 
			\[x_{\mathcal{J}}=\sqrt{\dfrac{n-1}{k(n-k)}}\sum_{i\in \mathcal{J}}\varphi_i\]
			Then the family $\{x_{\mathcal{J}}\}_{\mathcal{J}\in\mathcal{B}}$ forms a balanced, two-distance tight frame for $\RR^{n-1}$ at angles $\dfrac{xn-k^2}{k(n-k)}$ and $\dfrac{yn-k^2}{k(n-k)}$ with multiplicities $b-s-1$ and $s$, respectively.
		\end{enumerate}
	\end{theorem}
	\begin{proof}
		(1): It is clear that $x_{\mathcal{J}}$'s are unit norm vectors. If $|B_\mathcal{J}\cap B_\mathcal{J'}|=x$, then 
		\[\langle x_\mathcal{J}, x_\mathcal{J'}\rangle=\frac{1}{k}|\mathcal{J}\cap \mathcal{J'}|=\dfrac{x}{k}.\] Similarly, if $|\mathcal{J}\cap \mathcal{J'}|=y$, then 
		$\langle x_\mathcal{J}, x_\mathcal{J'}\rangle=\dfrac{y}{k}.$ Thus, $\{x_\mathcal{J}\}_{\mathcal{J}\in \mathcal{B}}$ is a two-distance set at angles $x/k$ and $y/k$. Moreover, by Theorem \ref{SRG}, it is regular and multiplicities of angles $x/k$ and $y/k$ are $b-s-1$ and $s$, respectively.
		
		Now we will show that $\{x_\mathcal{J}\}_{\mathcal{J}\in \mathcal{B}}$ is actually a frame for $\RR^n$. Let $F$ be its synthesis operator. Note that $\sqrt{k}F$ is the incidence matrix of the design. Hence,
		\[kFF^*=(r-\lambda)I+\lambda J.\]
	It follows that 
		\[\det(kFF^*)=(r-\lambda+n\lambda)(r-\lambda)^{n-1}>0.\] So $\{x_\mathcal{J}\}_{\mathcal{J}\in \mathcal{B}}$ is a frame for $\RR^n$. To see this frame is not tight, we compute its Grammian constant.
		\begin{align*}
		1+s\dfrac{y}{k}+(b-s-1)\dfrac{x}{k}&=1+\dfrac{k(r-1)-x(b-1)}{y-x}.\dfrac{y}{k}+\left[b-1-\dfrac{k(r-1)-x(b-1)}{y-x}\right]\dfrac{x}{k}\\
		&=1+\dfrac{y(r-1)}{y-x}-\dfrac{xy(b-1)}{k(y-x)}+\left[\dfrac{y(b-1)-k(r-1)}{y-x}\right]\dfrac{x}{k}\\
		&=1+\dfrac{y(r-1)}{y-x}-\dfrac{x(r-1)}{y-x}\\
		&=r\not=\dfrac{b}{n}.
		\end{align*}
		So this frame is not tight by Proposition \ref{pro2}.

		(2): Note that $\langle \varphi_i, \varphi_j\rangle=-\frac{1}{n-1}$, for $i\not=j$. We compute
		\[\langle \sum_{i\in \mathcal{J}}\varphi_i, \sum_{j\in \mathcal{J}}\varphi_j\rangle=\sum_{i=1}^k\|\varphi_i\|^2+2\sum_{i< j}\langle \varphi_i, \varphi_j\rangle=k-\dfrac{k(k-1)}{n-1}=\dfrac{k(n-k)}{n-1}.\]
		So $\{x_\mathcal{J}\}_{\mathcal{J}\in \mathcal{B}}$ is a set of unit norm vectors. 
		
		If $|\mathcal{J}\cap \mathcal{J'}|=x$, then 
		\begin{align*}
		\langle x_\mathcal{J}, x_\mathcal{J'}\rangle&=\dfrac{n-1}{k(n-k)}\langle \sum_{i\in \mathcal{J}}\varphi_i, \sum_{j\in \mathcal{J'}}\varphi_j\rangle\\
		&=\dfrac{n-1}{k(n-k)}\left[x-(k^2-x)\frac{1}{n-1}\right].\\
		&=\dfrac{xn-k^2}{k(n-k)}.
		\end{align*}
		
		Similarly, if $|\mathcal{J}\cap \mathcal{J'}|=y$, then 	
		$$\langle x_\mathcal{J}, x_\mathcal{J'}\rangle=\dfrac{yn-k^2}{k(n-k)}.$$
		Thus, $\{x_\mathcal{J}\}_{\mathcal{J}\in \mathcal{B}}$ is a regular two-distance set at angles and multiplicities as in the claim of the theorem. Moreover, since $\{\varphi_i\}_{i=1}^{n}$ is a simplex, it follows that  $\{x_\mathcal{J}\}_{\mathcal{J}\in \mathcal{B}}$ is a balanced tight frame by Proposition \ref{pro3}.
	\end{proof}
	
	In the following, we will construct some infinite families of regular two-distance frames with large cardinalities.
	\begin{example}\label{ex2}
		For $n\in \mathbb{N}$, let $V=[n]$ and $\mathcal{B}$ be the set of all 2-element subsets of $V$. Then $(V, \mathcal{B})$ is a quasi-symmetric design with parameters
		\[v=n, b=\frac{(n-1)n}{2}, r=n-1, k=2, \lambda=1 ,\]
		and intersection numbers $x=0, y=1$.
		
		By (1) of Theorem \ref{construction1}, we get a non-tight regular two-distance frame of $m=\frac{(n-1)n}{2}$ vectors for $\RR^n$ at angles $\alpha=0$ and $\beta=1/2$, with multiplicities $\frac{n^2-5n+6}{2}$ and $2n-4$, respectively. This frame achieves the upper bound for cardinality of regular two-distance sets which have two non-negative angles as shown in Corollary \ref{posi angles}.
		
		On the other hand, if we apply (2) of Theorem \ref{construction1}, we will get a balanced, two-distance tight frame of $\frac{(n-1)n}{2}$ vectors for $\RR^{n-1}$. Its angles are $-\frac{2}{n-2}$ and $\frac{n-4}{2(n-2)}$ with respective multiplicities $\frac{n^2-5n+6}{2}$ and $2n-4$. This is a maximal two-distance set if $n\not=(2\ell+1)^2-2$ for all $\ell\in \mathbb{N}$.
	\end{example}

We have seen that we can use quasi-symmetric designs to construct regular two-distance frames. Now we will see that the frames constructed in this way have interesting structure. 

\begin{theorem}\label{structure}
	Let $X=\{x_i\}_{i=1}^m$ be the regular two-distance frame constructed by (1) or (2) of Theorem \ref{construction1} at angles $\alpha$ and $\beta.$ Denote by $\mathcal{I}^\alpha$ and  $\mathcal{I}^\beta$ the sets:
	\[\mathcal{I}^\alpha=\{i\in [m]: \langle x_i, x_1\rangle=\alpha\}, \quad \mathcal{I}^\beta=\{i\in [m]: \langle x_i, x_1\rangle=\beta\}.\]
	Then the either set $Y=\{x_i\}_{i\in \mathcal{I}^\alpha}$ (similarly, the set $Z=\{x_i\}_{i\in \mathcal{I}^\beta}$) is a regular two-distance set at angles $\alpha$ and $\beta$ or it satisfies $\langle x_i, x_j\rangle=\alpha$, for all $i, j\in \mathcal{I}^\alpha, i\not=j$ (respectively, $\langle x_i, x_j\rangle=\beta$, for all $i, j\in \mathcal{I}^\beta, i\not=j$). 
\end{theorem}
\begin{proof}
		Let $([n], \mathcal{B})$ be the $2$-$(n, b, r, k, \lambda)$ quasi-symmetric  design with intersection numbers $x$ and $y$ used to construct the frame $X$.  By the construction of $X$, we can associate each frame vector to the corresponding vertex of the block graph of the design. We will say that two vectors are adjacent if their associated vertexes are adjacent. As we have seen in the proof of Theorem \ref{construction1}, the angles $\alpha$ and $\beta$ are determined by the intersection numbers $x$ and $y$. We can assume that $\alpha$ is determined by $y$. Thus, the set $\{x_i\}_{i\in\mathcal{I}^{\alpha}}$ contains all adjacent vectors of $x_1$. By Theorem \ref{SRG}, the block graph is strongly regular with parameters $(b, s, \mu_1, \mu_2)$. It follows that the set $\{x_i\}_{i\in \mathcal{I}^\alpha}$ has $s$ elements and each row of the Gram matrix of $Y$ has exactly $\mu_1$ elements $\alpha$. Thus, if $\mu_1<s-1$, then $X$ is a regular two-distance set, otherwise, we have that $\langle x_i, x_j\rangle=\alpha$, for all $i, j\in \mathcal{I}^\alpha, i\not=j$. With a similar argument, we get the claim for the set $Z$. This completes the proof.	
\end{proof}
\begin{example}\label{new}
	Let $n=5$ and let $X=\{x_i\}_{i=1}^{10}$ be the balanced, two-distance tight frame for $\RR^4$ constructed by Example \ref{ex2}. Then the angles of $X$ are $1/6$ and $-2/3$ with multiplicities $6$ and $3$, respectively. If we denote by $\{\varphi_i\}_{i=1}^5$ the simplex in $\RR^4$ and enumerate the vectors of $X$ as 
	\[x_1=\frac{\sqrt{6}}{3}(\varphi_1+\varphi_2), \quad  x_2=\frac{\sqrt{6}}{3}(\varphi_1+\varphi_3),\quad x_3=\frac{\sqrt{6}}{3}(\varphi_1+\varphi_4),\quad x_4=\frac{\sqrt{6}}{3}(\varphi_1+\varphi_5),\]
	\[x_5=\frac{\sqrt{6}}{3}(\varphi_2+\varphi_3),\quad x_6=\frac{\sqrt{6}}{3}(\varphi_2+\varphi_4),\quad
	x_7=\frac{\sqrt{6}}{3}(\varphi_2+\varphi_5), \quad  x_8=\frac{\sqrt{6}}{3}(\varphi_3+\varphi_4),\] \[x_9=\frac{\sqrt{6}}{3}(\varphi_3+\varphi_5),\quad
	x_{10}=\frac{\sqrt{6}}{3}(\varphi_4+\varphi_5),\]
	then the Gram matrix of $X$ is as follows.
	\[\begin{bmatrix}
	1&1/6&1/6&1/6&1/6&1/6&1/6&-2/3&-2/3&-2/3\\
	1/6&1&1/6&1/6&1/6&-2/3&-2/3&1/6&1/6&-2/3\\
	1/6&1/6&1&1/6&-2/3&1/6&-2/3&1/6&-2/3&1/6\\
	1/6&1/6&1/6&1&-2/3&-2/3&1/6&-2/3&1/6&1/6\\
	1/6&1/6&-2/3&-2/3&1&1/6&1/6&1/6&1/6&-2/3\\
	1/6&-2/3&1/6&-2/3&1/6&1&1/6&1/6&-2/3&1/6\\
	1/6&-2/3&-2/3&1/6&1/6&1/6&1&-2/3&1/6&1/6\\
	-2/3&1/6&1/6&-2/3&1/6&1/6&-2/3&1&1/6&1/6\\
	-2/3&1/6&-2/3&1/6&1/6&-2/3&1/6&1/6&1&1/6\\
	-2/3&-2/3&1/6&1/6&-2/3&1/6&1/6&1/6&1/6&1
	\end{bmatrix} 
	\]
	Let $Y=\{x_i\}_{i=2}^7$ and $Z=\{x_i\}_{i=8}^{10}$. Then $Y$ is a regular two-distance set at angles $1/6$ and $-2/3$ with respective multiplicities $3$ and $2$. The set $Z$ satisfies $\langle x_i, x_j\rangle=1/6$ for $i\not=j$.
\end{example}
\begin{remark} Not every regular two-distance set has the property in Theorem \ref{structure}. We will see this by the following example.
\end{remark}
\begin{example}
	Let the set of vectors $X=\{x_i\}_{i=1}^8$ to be the columns of the matrix:
		\[\begin{bmatrix}
		\sqrt{6}/3&0&-\sqrt{6}/3&0&0&0&0&0\\
		-\sqrt{3}/3&\sqrt{3}/3&-\sqrt{3}/3&\sqrt{3}/3&-\sqrt{3}/3&\sqrt{3}/3&-\sqrt{3}/3&\sqrt{3}/3\\
		0&\sqrt{6}/3&0&-\sqrt{6}/3&0&0&0&0\\
		0&0&0&0&\sqrt{6}/3&0&-\sqrt{6}/3&0\\
		0&0&0&0&0&\sqrt{6}/3&0&-\sqrt{6}/3
		\end{bmatrix} \]
\end{example}
The Gram matrix of $X$ is
	\[\begin{bmatrix}
	1&-1/3&-1/3&-1/3&1/3&-1/3&1/3&-1/3\\
	-1/3&1&-1/3&-1/3&-1/3&1/3&-1/3&1/3\\
	-1/3&-1/3&1&-1/3&1/3&-1/3&1/3&-1/3\\
	-1/3&-1/3&-1/3&1&-1/3&1/3&-1/3&1/3\\
	1/3&-1/3&1/3&-1/3&1&-1/3&-1/3&-1/3\\
	-1/3&1/3&-1/3&1/3&-1/3&1&-1/3&-1/3\\
	1/3&-1/3&1/3&-1/3&-1/3&-1/3&1&-1/3\\
	-1/3&1/3&-1/3&1/3&-1/3&-1/3&-1/3&1
	\end{bmatrix} \]
Then $$Y=\{x_i : \langle x_i, x_1\rangle=-1/3\}=\{x_2, x_3, x_4, x_6, x_8\}.$$ From the Gram matrix of $X$, we see that $X$ is a regular two-distance set but $Y$ is not regular.

It has been seen that all known maximal spherical two-distances set are tight frames. However, in the following, we will give for the first time an example of a non-tight maximal two-distance frame. 
\begin{example}
	Let $X=\{x_i\}_{i=1}^{10}$ be the frame for $\RR^4$ given in Example \ref{new}. Let $P$ be the orthogonal projection onto $x_1$. Then 
	\[Px_i=\langle x_i, x_1\rangle x_1=1/6x_1, \mbox{ for } i=2, 3, \ldots,7.\]
	For $i, j =2, 3, \ldots, 7$, we have
	\[\langle(I-P)x_i, (I-P)x_j\rangle=\langle x_i, x_j\rangle-\langle Px_i, Px_j\rangle=\langle x_i, x_j\rangle-1/36.\]
	This implies that for all $i=2, 3, \ldots, 7$, $y_i:=\frac{(I-P)x_i}{\|(I-P)x_i\|}=\sqrt{\frac{36}{35}}(I-P)x_i$ is a unit norm vector in $\RR^3$. Moreover
	\[\langle y_i, y_j\rangle=\dfrac{36}{35}\left(\langle x_i, x_j\rangle-\dfrac{1}{36}\right).\]
	From the Gram matrix of $X$, we see that $Y=\{y_i\}_{i=2}^7$ is regular two-distance frame for $\RR^3$ at angles $1/7$ and $-5/7$ with respective multiplicities $3$ and $2$. By Corollary \ref{co1}, 
	$Y$ is not tight since
	\[1+3(1/7)^2+2(-5/7)^2=102/49\not=6/3.\] 
	
	In order to see the frame vectors of $Y$, we can use the simplex in $\RR^4$ (the columns are the vectors, see the construction of the simplex in \cite{MP}):
	\[\begin{bmatrix}
	-\sqrt{10}/4&\sqrt{10}/4&0&0&0\\
	-\sqrt{30}/12&-\sqrt{30}/12&\sqrt{30}/6&0&0\\
	-\sqrt{15}/12&-\sqrt{15}/12&-\sqrt{15}/12&\sqrt{15}/4&0\\
	-1/4&-1/4&-1/4&-1/4&1\\
	\end{bmatrix} 
	\] to construct the frame vectors of $X$.
	Then we get the following frame $Y$:
	\[\begin{bmatrix}
	-\sqrt{21}/7&-\sqrt{21}/7&-\sqrt{21}/7&\sqrt{21}/7&\sqrt{21}/7&\sqrt{21}/7\\
	4\sqrt{7}/21&-2\sqrt{7}/21&-2\sqrt{7}/21&4\sqrt{7}/21&-2\sqrt{7}/21&-2\sqrt{7}/21\\
	-5\sqrt{14}/42&\sqrt{14}/6&-\sqrt{14}/21&-5\sqrt{14}/42&\sqrt{14}/6&-\sqrt{14}/21\\
	-\sqrt{210}/42&-\sqrt{210}/42&\sqrt{210}/21&-\sqrt{210}/42&-\sqrt{210}/42&\sqrt{210}/21
	\end{bmatrix} 
	\]
	These column vectors are perpendicular to the row vector: $$(0, -\sqrt{30}, -\sqrt{15}, -3)$$ so they are in $\RR^3$. This frame is a maximal, regular two-distance set, summing to zero but it is not tight. 
\end{example}

	The following theorem gives another simple construction of regular two-distance sets. The interesting thing here is that we can construct infinitely many of them with the same number of vectors in the same dimension, which cannot happen for the tight frame case.
	\begin{theorem}\label{construction2}
		Let $\{x_i\}_{i=1}^m$ be a regular two-distance set at angles $\alpha, \beta$ with multiplicities $k_\alpha, k_\beta$, respectively. Assume that $X$ is not balanced. Denote by $\bar{x}=\sum_{i=1}^{m}x_i\not=0$, and $c$ its Grammian constant. Then for each $t\in \RR$, $\left\{y_i=\frac{1}{\sqrt{{1+2tc+t^2mc}}}(x_i+t\bar{x})\right\}_{i=1}^m$ is a regular two-distance set at angles $\frac{\alpha+2tc+t^2mc}{1+2tc+t^2mc}$ and $\frac{\beta+2tc+t^2mc}{1+2tc+t^2mc}$ with respective multiplicities $k_\alpha$ and $k_\beta$. This set is not balanced unless $t=-1/m$.
	\end{theorem}
	\begin{proof}
		For any $i, j$, we have that
		\begin{align*}
		\langle x_i+t\bar{x}, x_j+t\bar{x}\rangle&=\langle x_i, x_j\rangle+t\langle x_i, \bar{x}\rangle+t\langle x_j, \bar{x}\rangle+t^2\langle \bar{x}, \bar{x}\rangle\\
		&=\langle x_i, x_j\rangle+2tc+t^2mc.
		\end{align*}
		This shows that $y_i's$ are unit norm and $Y$ is a regular two-distance set with angles as in the claim. This set is not balance since 
		\[\sum_{i=1}^{m} y_i=\frac{1}{\sqrt{{1+2tc+t^2mc}}}\sum_{i=1}^{m}(x_i+t\bar{x})=\frac{tm+1}{\sqrt{{1+2tc+t^2mc}}}\bar{x}\not=0,\]
		by our assumption.
	\end{proof}
	
	\begin{example} \label{exam1}
		Let $X=\{x_i\}_{i=1}^{\frac{(n-1)n}{2}}$ be the non-balanced, regular two-distance set for $\RR^n$ constructed in Example \ref{ex2}. Let $Y$ be the two-distance set constructed from $X$ by Theorem \ref{construction2}. Then for $t$ large enough, both angles of $Y$ are positive. Note that the cardinality of $Y$ attains the upper bound in Corollary \ref{posi angles}.	
	\end{example}
	
	\begin{corollary}
		There exist infinitely many regular two-distance sets with the same number of vectors in the same dimension. 
	\end{corollary}
	\begin{proof}
		Again, let $X=\{x_i\}_{i=1}^{\frac{(n-1)n}{2}}$ be the regular two-distance frame for $\RR^n$ constructed in Example \ref{ex2}. This set is not balanced, so the claim follows by Theorem \ref{construction2}.
	\end{proof}

	The following result can be deduced from the correspondence between ETFs and a class of strongly regular graphs, see for example \cite{W0}. However, for the completeness of the paper, we will give a simple proof for this.
	\begin{proposition}\label{pro6}
		Let $X=\{x_i\}_{i=1}^m$ be an ETF for $\RR^n$ ($m>n+1$) at angle $\alpha$, meaning $|\langle x_i, x_j\rangle|=\alpha$, for all $i\not=j$. Assume that $\langle x_i, x_1\rangle =\alpha$, for all $i\geq 2$. Then $Y=\{x_i\}_{i=2}^m$ is a regular two-distance set for $\RR^n$ at angles $\alpha$ and $-\alpha$. Moreover, $Y$ is not balanced.
	\end{proposition}
	\begin{proof}
		Let $G$ be the Gram matrix of $X$. Then $G$ is self-adjoint and all entries in the first row other than $G_{11}$ are $\alpha$. For a fixed row $i\geq 2$, let $k_\alpha$ be the number of $\alpha$ in this row. Since $G^2=\frac{m}{n}G$, it follows that
		\[2\alpha+(k_\alpha-1)\alpha^2-(m-k_\alpha-1)\alpha^2=\frac{m}{n}\alpha,\]
		which is equivalent to 
		\[k_\alpha=\frac{m}{2}+\frac{m-2n}{2n\alpha}.\]
		Thus $k_\alpha$ does not depend on the row $i$, so the claim follows.
		
		In order to see $Y$ is not balanced, we assume by way of contradiction that $\sum_{i=2}^{m}x_i=0$. Then $$0=\langle x_1, \sum_{i=2}^{m}x_i\rangle=(m-1)\alpha,$$ which cannot happen.	
	\end{proof}
	
It is known that the existence of ETFs of $2n$ vectors for $\RR^n$ implies $n$ is odd and $(2n-1)$ is the sum of two squares, see \cite{FM, STDH}. We will state a part of this result here as a simple consequence of our results.
	\begin{proposition}
		If there exists an ETF of $2n$ vectors for $\RR^n$, then $n$ must be odd.
	\end{proposition}
	\begin{proof}
		Suppose $X=\{x_i\}_{i=1}^{2n}$ is an ETF for $\RR^n$ at angle $\alpha$. We can assume that $\langle x_i, x_1\rangle=\alpha$ for all $i\geq 2$. By Proposition \ref{pro6}, $Y=\{x_i\}_{i=2}^{2n}$ is a regular two distance set at angles $\alpha$ and $-\alpha$. Note that the multiplicity of the angle $\alpha$ is $k_\alpha-1=n-1$, where $k_\alpha$ is computed as in the proof of Proposition \ref{pro6}. Since the cardinality of $Y$ is odd, by Proposition \ref{even multip}, $n-1$ is even, which is the claim.
	\end{proof}
	
	Example \ref{exam1} gave a family of maximal regular two-distance sets with two positive angles for any dimensions $n\not=(2k+1)^2-2, k\in \mathbb{N}$. Now we will present example of such sets in dimensions of form $n=(2k+1)^2-2$.
\begin{example}\label{non-tight maximal}
		Let $X=\{x_i\}_{i=1}^{\frac{n(n+1)}{2}}$ be an ETF in $\RR^n$ and suppose $\langle x_i, x_1\rangle=\alpha>0$ for all $i\geq 2$. Let $\bar{x}=\sum_{i=2}^{\frac{n(n+1)}{2}}x_i$. Then for each $t>0$, by Theorem \ref{construction2}, the normalized of the vectors $\{x_i+t\bar{x}\}_{i=2}^{\frac{n(n+1)}{2}}$ form a non-balanced, regular two-distance set of $\frac{n(n+1)}{2}-1=\frac{(n-1)(n+2)}{2}$ vectors in $\RR^n$ at angles shown in the theorem. Now let $t$ be large enough we get the set with two positive angles. This set is maximal by Corollary \ref{posi angles}.
\end{example}

Given an ETF, we can construct a balanced, two-distance tight frame as in the following.

\begin{theorem}\label{construction3}
	Let $X=\{x_i\}_{i=1}^m$ be an ETF for $\RR^n$ ($m>n+1$) at angle $\alpha$. Assume that $\langle x_i, x_1\rangle =\alpha$, for all $i\geq 2$. Let $P$ be the orthogonal projection onto $\spn\{x_1\}$. Then $Y=\left\{\frac{(I-P)x_i}{\|(I-P)x_i\|}\right\}_{i=2}^m$ is a balanced, two-distance tight frame for the space $x_1^\perp$ at angles $\frac{\alpha}{1+\alpha}$ and $\frac{-\alpha}{1-\alpha}$. 
\end{theorem} 
\begin{proof}
	We have that
	\[Px=\langle x, x_1\rangle x_1, \mbox{ for all } x\in \RR^n.\] Hence, for all $i, j\geq 2$, 
	\[\langle (I-P)x_i, (I-P)x_j=\langle x_i, x_j\rangle-\langle x_i, x_1\rangle\langle x_j, x_1\rangle=\langle x_i, x_j\rangle-\alpha^2.\]
	It follows that $\|(I-P)x_i\|=\sqrt{1-\alpha^2}$ for all $i\geq 2$, and $Y$ is a two-distance tight frame with angles as in the claim. 
	
	To see $Y$ is balanced, we note that
	\[\frac{m}{n}x_1=\sum_{i=1}^{m}\langle x_1, x_i\rangle x_i=x_1+\alpha\sum_{i=2}^{m}x_i,\]
	which implies 
	\[\sum_{i=2}^{m}x_i=\frac{m-n}{n\alpha}x_1.\]
	Therefore,
	\begin{align*}
	\sum_{i=2}^{m}(I-P)x_i&=\sum_{i=2}^{m}x_i-\sum_{i=2}^{m}Px_i\\
	&=\frac{m-n}{n\alpha}x_1-\sum_{i=2}^{m}\alpha x_1=\left[\frac{m-n}{n\alpha}-(m-1)\alpha\right]x_1. 
	\end{align*}
	By the definition of $P$, $	\sum_{i=2}^{m}(I-P)x_i$ must be zero.
\end{proof}
\begin{remark}
	From the proof above, we get that the value for $\alpha$ is $\sqrt{\frac{m-n}{n(m-1)}}$, which is called the Welch bound, see for example in \cite{FM}. The multiplicity for the angle $\frac{\alpha}{1+\alpha}$ is $k_\alpha-1$, where $k_\alpha$ is computed as in the proof of Proposition \ref{pro6}. 
\end{remark}


From a regular two-distance set in $\RR^n$, we can lift it to an another regular two-distance set in one higher dimension with desired angles or Grammian constant.
\begin{theorem}\label{desire angles}
	Let $X=\{x_i\}_{i=1}^m$ be a regular two-distance set in $\RR^n$ at angles $\alpha$ and $\beta$ with multiplicities $k_\alpha$ and $k_\beta$. Let $c$ be its Grammian constant.
	\begin{enumerate}
		\item For any $\alpha\le \alpha'\leq 1$, there is a $m$-element regular two-distance set in $\RR^{n+1}$ such that $\alpha'$ is one of its angles with multiplicity $k_\alpha$.
		\item  For any $c\le c' <m$, there is a $m$-element regular two-distance set in $\RR^{n+1}$ with Grammian constant $c'$.
	\end{enumerate}
\end{theorem}

\begin{proof}(1)  
	We define 
	\[ y_i=(tx_i,\sqrt{1-t^2}),\mbox{ for all }i\in [m].\]
	Then 
	\[ \langle y_i,y_j\rangle = \begin{cases} 
	t^2\alpha+1-t^2\mbox{ if }\langle x_i,x_j\rangle =\alpha\\
	t^2\beta+1-t^2\mbox{ if }\langle x_i,x_j\rangle =\beta.
	\end{cases} \]
	The continuous function $f(t)=t^2\alpha+1-t^2$ equals 1 when $t=0$ and equals $\alpha$ when $t=1$. So we can choose $t$ so that $f(t)=\alpha'$.
	
	(2)  Let $y_i's$ be as above. Since the Grammian constant of $X$ is $c$, we have
	\[1+k_\alpha \alpha+k_\beta\beta=c.\]
	Hence, 
	\begin{align*}
		1+k_\alpha(t^2\alpha+1-t^2)+k_\beta(t^2\beta+1-t^2) &= 1+t^2(k_\alpha\alpha+k_\beta\beta)+(1-t^2)(k_\alpha+k_\beta)\\
		&=1+ t^2(c-1)+(1-t^2)(m-1)\\
		&= t^2c+(1-t^2)m.
	\end{align*}
Thus, $t^2c+(1-t^2)m$ equals $c'$ precisely when $t=\sqrt{\frac{m-c'}{m-c}}$.
\end{proof}

\begin{remark}
	If we do not impose any condition on angles or Grammian constants, then for each $t$, the set $\{y_i\}_{i=1}^m$ constructed in the proof of Theorem \ref{desire angles} is a regular two-distance set in $\RR^{n+1}$. This gives another way to construct infinitely many regular two-distance sets from a given one.
\end{remark}
\section{Connection with equiangular lines}

A set of lines in Euclidean space is called equiangular, if the angle between
each pair of lines is the same. In other words, if we choose a unit vector that spans each line, then this set of vectors forms a spherical two-distance set at angles $\{\alpha, -\alpha\}$ for some $\alpha \in [0, 1)$.

It has been shown that from a spherical two-distance set of $m$ vectors in $\RR^n$ with angles $\alpha, \beta$ satisfying $\alpha+\beta<0$, we can construct a set of $m$ equiangular lines in $\RR^{n+1}$, see \cite{DJS, GY}. We will restate it as in the following proposition.
\begin{proposition}\label{pro5}
	Let $\{x_i\}_{i=1}^m$ be a spherical two-distance set in $\RR^n$ at angles $\alpha, \beta$ such that $\alpha+\beta<0$. Then there exist $m$ equiangular lines in $\RR^{n+1}$, and they can be constructed explicitly.
\end{proposition}

\begin{proof}
		We define a set of unit vectors in $\RR^{n+1}$ by
		\[ y_i=(tx_i,\sqrt{1-t^2}),\mbox{ for all }i\in [m].\]
		Then 
		\[ \langle y_i,y_j\rangle = \begin{cases} 
		t^2\alpha+1-t^2\mbox{ if }\langle x_i,x_j\rangle =\alpha\\
		t^2\beta+1-t^2\mbox{ if }\langle x_i,x_j\rangle =\beta.
		\end{cases} \]
		Now let $t=\sqrt{\frac{2}{2-(\alpha+\beta)}}$, then $t^2\alpha+1-t^2=-(t^2\beta+1-t^2)$. Thus, these vectors $\{y_i\}_{i=1}^m$ define a set of $m$ equiangular lines in $\RR^{n+1}$.
\end{proof}

Several examples of maximal equiangular lines in low dimensions can be constructed by the method of Proposition \ref{pro5}.

\begin{example} [see also in \cite{G}]
	Recall in Example \ref{ex2}, we constructed maximal two-distance tight frames of $\frac{(n-1)n}{2}$ vectors in $\RR^{n-1}$ at angles $\alpha=-\frac{2}{n-2}$ and $\beta=\frac{n-4}{2(n-2)}$. We see that $\alpha+\beta<0$ if $n<8$. Thus, for dimensions less than 8, we can use these two-distance frames to construct equiangular lines in spaces of one higher dimensions.  In particular, we obtain maximal equiangular lines in dimensions 4, and 5 with 6 and 10 lines, respectively.
\end{example}

Now we give another way to construct equiangular lines from two-distance sets but in the same space.
\begin{proposition}\label{lastpro}
	Let $\{x_i\}_{i=1}^m$ be a regular two-distance set in $\RR^n$ at angles $\alpha, \beta$ in $\RR^n$ and the Grammian constant $c>0$. If $\alpha+\beta\leq 2c/m$, then there exist $m$ equiangular lines in $\RR^{n}$.
\end{proposition}
\begin{proof}
	Let $\bar{x}=\sum_{i=1}^{m}x_i$. By Theorem \ref{construction2},  for each $t\in \RR$, the set $$\left\{y_i=\frac{1}{\sqrt{{1+2tc+t^2mc}}}(x_i+t\bar{x})\right\}_{i=1}^m$$ is a regular two-distance set at angles $\frac{\alpha+2tc+t^2mc}{1+2tc+t^2mc}$ and $\frac{\beta+2tc+t^2mc}{1+2tc+t^2mc}$. This set spans a set of equiangular lines if the equation
	$$\frac{\alpha+2tc+t^2mc}{1+2tc+t^2mc}=-\frac{\beta+2tc+t^2mc}{1+2tc+t^2mc}$$ has solution in $t$, which is equivalent to $\alpha+\beta\leq 2c/m$. This completes the proof.
	 
\end{proof}
\begin{corollary}\label{cor3}
	Let $([n], \mathcal{B})$ be a $2$-$(n, b, r, k, \lambda)$ quasi-symmetric design with intersection numbers $x$ and $y$. If $x+y\leq 2kr/b$, then there exist $b$ equiangular lines in $\RR^{n}$.
\end{corollary}
\begin{proof}
	By Theorem \ref{construction1}, there exist a regular two-distance set of $b$ vectors in $\RR^n$ at angles $x/k$ and $y/k$. Note that the Grammian constant of this set is $r>0$ as shown in the proof of the theorem. The result then follows by Proposition \ref{lastpro}.
\end{proof}
\begin{remark}
	Corollary \ref{cor3} gives a sufficient condition to construct equiangular lines from quasi-symmetric designs. Note that the author in \cite{G} (Theorem 6.2) gave a more general construction of equiangular lines using block sets which are not necessary quasi-symmetric designs. Our theorem and proof are not identical because we discovered it independenty.  But, the two are fundamentally the same.
\end{remark}
\begin{example}[see also in \cite{ G}]
	Using quasi symmetric design with parameters $(v, b, r, k, \lambda)=(22, 176, 56, 7, 16)$, and intersection numbers $x=1, y=3$ which exists (see \cite{CD}), we get 176 lines in $\mathbb{R}^{22}$, which is maximal. 
\end{example}
Relying on the known-bound for the maximal number of equiangular lines in certain dimensions, we can use Corollary \ref{cor3} to confirm the non-existence of some quasi-symmetric designs.

\begin{example}
	There is no quasi-symmetric design of parameters $(v, b, r, k, \lambda)=(9, 36, 20, 5, 10)$ with intersection numbers $x=1, y=3$ since the maximal number of equiangular lines in $\mathbb{R}^9$ is $28<b=36$. Likewise, since the maximal number of equiangular lines in $\mathbb{R}^{19}$ is on the range $72-75$, there does not exist a quasi-symmetric design of parameters $(v, b, r, k, \lambda, x, y)=(19, 76, 36, 9, 16, 3, 5)$.
\end{example}

In the following, we will make another connection between quasi-symmetric designs and equiangular lines. 

It has been shown that the maximum number of equiangular lines has not been established yet even in some low dimensions. For instance, in dimensions $n=42, 45$ and $46$, the bound is on the range $276-288$, $344-540$ and $344-736$, respectively, see the table 12.3 page 289 in \cite{W}. 

The existence of the following quasi-symmetric designs has not been confirmed yet, see \cite{CD}. However, if they exist, then we will have more information about the maximal number of equiangular lines in dimensions mentioned above. Namely, the existence of a  quasi-symmetric design with parameters $(v, b, r, k, \lambda, x, y)=(42, 287, 123, 18, 51, 6, 9)$ ensures the existence of 287 equiangular lines in $\RR^{42}$. Also, a quasi-symmetric design with $(v, b, r, k, \lambda, x, y)=(45, 396, 132, 15, 42, 3, 6)$ can be used to constructed 396 equiangular lines in $\RR^{45}$. Finally, if a quasi-symmetric design with $(v, b, r, k, \lambda, x, y)=(46, 621, 216, 16, 72, 4, 7)$ exists, then there are at least 621 equiangular lines  in $\RR^{46}$.

In \cite{G} it is shown that quasi-symmetric designs naturally appear from the structure of certain equiangular tight frames.  This result is used to classify the known sets of ETFs that saturate the absolute bound $n(n+1)/2$.

As a final result of the paper, we will give some new necessary conditions for the existence of quasi-symmetric designs with parameters $(v, b, r, k, \lambda, x, y)$.

\begin{proposition} Suppose there exists a quasi-symmetric design with parameters $(v, b, r, k, \lambda, x, y)$. Then we have the following:
	\begin{enumerate}
		\item $x\leq \frac{k^2}{v}\leq y$, and the equality cannot happen in both places at the same time. Moreover, if $x=\frac{k^2}{v}$ or $y=\frac{k^2}{v}$, then $b$ is multiple of $b-v+1$. Namely, $b=(s+1)(b-v+1)$ if $x=\frac{k^2}{v}$, and $b=(b-s-1)(b-v+1)$ if $y=\frac{k^2}{v}$, where $s=\dfrac{k(r-1)-x(b-1)}{y-x}$.
		
		\item If $b$ is odd then $s=\dfrac{k(r-1)-x(b-1)}{y-x}$ is even. In particular $k(r-1)$ must be even for all possible values of the intersection numbers $x, y$.
	\end{enumerate}
\end{proposition}
\begin{proof}
	
	(1): By (2) of Theorem \ref{construction1}, we get a balanced, regular two-distance tight frame for $\RR^{v-1}$ at angles $\dfrac{xv-k^2}{k(n-k)}$ and $\dfrac{yv-k^2}{k(n-k)}$ with multiplicities $b-s-1$ and $s$, respectively. By Theorem \ref{angles of tight}, this frame must have one non-negative angle and one non-positive angle. Hence, $x\leq \frac{k^2}{v}\leq y$.
	
	Now we assume that $x=k^2/v$. Then by Proposition \ref{zero_angle}, the Naimark complement of this frame is $(s+1)$ copies of an orthonormal basis for $R^{b-(v-1)}$, which is the claim. Similarly, we get the result for the case $y=k^2/v$.
	
	(2): Again, let $Y$ be the frame constructed by (2) of Theorem \ref{construction1}. Since the number of vectors of $X$ is odd, the multiplicities must be even by Proposition \ref{even multip}. 
\end{proof}

\begin{remark}
	It has been shown that $k^2/v<\lambda$, for example see \cite{P}. However, we observed that most quasi-symmetric designs have $y$ much smaller than $\lambda$. So the condition $\frac{k^2}{v}\leq y$ is somewhat stronger.  
\end{remark}


\begin{thebibliography}{WW}
	
\bibitem{BGOY}	A. Barg, A. Glazyrin, K.A. Okoudjou, W.-H. Yu, {\it Finite two-distance tight frames}, Linear Algebra Appl. 475 (2015) 163-175.

\bibitem{BF} J. J. Benedetto, M. Fickus,  {\it Finite normalized tight frames}, Adv. Comput. Math. 18 (2003) 357-385

\bibitem{CG} Peter G. Casazza and Gitta Kutyniok, editors. Finite frames. Applied and Numerical Harmonic Analysis. Birkhauser/Springer, New York, 2013.

\bibitem{CFHT}  P.G. Casazza, A. Farzannia, J.I. Haas, and T.T. Tran, {\it Toward the classification of biangular harmonic frames}, 
Comput. Harmon. Anal. 46 (2019) no. 3, 544-568.

\bibitem{CD} C. J. Colbourn and J. H. Dinitz, editors. Handbook of combinatorial designs. Discrete Mathematics and its Applications (Boca Raton). Chapman Hall/CRC, Boca Raton, FL, second edition, 2007.

\bibitem{CHS} J. H. Conway, R. H. Hardin, and N. J. A. Sloane,  {\it Packing lines, planes, etc.: packings in Grassmannian spaces} Experiment. Math., 5(2):139-159, 1996.

\bibitem{DJS}  P. Delsarte, J.M. Goethals, J.J. Seidel, {\it Spherical codes and designs}, Geom. Dedicata 6 (1977) 363-388.

\bibitem{FMT} M. Fickus, D. G. Mixon, J. C. Tremain, {\it Steiner equiangular tight frames}, Linear Algebra Appl. 436 (2012) 1014-1027. 

\bibitem{FM} Matthew Fickus and Dustin G. Mixon,  {\it Tables of the existence of equiangular tight frames}, arXiv e-print, arXiv:1504.00253, 2016.

\bibitem{FJMPW} M. Fickus, J. Jasper, D. G. Mixon, J. D. Peterson, C. E. Watson, {\it Equiangular tight frames with centroidal symmetry}, Appl. Comput. Harmon. Anal. 44 (2018), no. 2, 476-496.

\bibitem{G} N. I. Gillespie, {\it Equiangular lines, incoherent sets and quasi-symmetric designs},  arXiv e-print, arXiv:1809.05739, 2018.

\bibitem{GY}	A. Glazyrin, W.-H. Yu, {\it Upper bounds for $s$-distance sets and equiangular lines}, Advances in Mathematics. 330 (2018) 810-833.

\bibitem{IS} Y. J. Ionin and M. S. Shrikhande, Combinatorics of Symmetric Designs, Cambridge Univ. Press, Cambridge, UK, 2006.



\bibitem{L}  P. Lisonek, {\it New maximal two-distance sets}, J. Combin. Theory Ser. A 77 (1997) 318-338.

\bibitem{MP} V. N. Malozemov, A. B. Pevnyi, {\it Equiangular tight frames}, Journal of Mathematical Sciences, Vol. 157, No. 6, 2009.

\bibitem{M}  O. Musin, {\it Spherical two-distance sets}, J.Combin. Theory Ser. A 116(4) (2009) 988-995.

\bibitem{P} R.M. Pawale, Quasi-symmetric designs with the difference of block intersection numbers two, Des. Codes Crypto. 58 (2011) no. 2, 111-121.

\bibitem{R} R. A. Rankin, {\it The closest packing of spherical caps in n dimensions}, Proc. Glasgow Math. Assoc. 2 (1955) 139-144.

\bibitem{STDH} M. A. Sustik, J. A. Tropp, I. S. Dhillon, R. W. Heath, {\it On the existence of equiangular tight frames}, Linear Algebra Appl. 426 (2007) 619-635.

\bibitem{W0} S. Waldron, {\it On the construction of equiangular frames from graph}, Linear Algebra Appl. 431(11) (2009) 2228-2242.

\bibitem{W} S. F. D Waldron. An introcduction of finite tight frames. Applied and Numerical Harmonic Analysis. Birkhauser/Springer, New York, 2018.

\bibitem{XZG} P. Xia, S. Zhou, and G. B. Giannakis, {\it Achieving the Welch bound with difference sets}, IEEE Trans. Inform. Theory, 51(5) (2005) 1900-1907.

\end{thebibliography}
\end{document}